\documentclass{amsart}
\usepackage{graphicx} 
\usepackage{amsmath}
\usepackage{amssymb}
\usepackage{comment}
\usepackage{amsthm}
\usepackage{amsfonts}
\usepackage{color}
\usepackage[colorlinks,citecolor=blue,urlcolor=blue,linkcolor=blue]{hyperref}

\theoremstyle{definition}

\newtheorem{mydef}{Definition}[section]

\newtheorem{myque}[mydef]{Question}

\newtheorem{mysetup}[mydef]{Setup}
\newtheorem{myconv}[mydef]{Convention}

\theoremstyle{remark}

\newtheorem{myrem}[mydef]{Remark}

\theoremstyle{plain}

\newtheorem{mycol}[mydef]{Corollary}
\newtheorem{mysen}[mydef]{Theorem}
\newtheorem{mylem}[mydef]{Lemma}

\newtheorem{myfact}[mydef]{Fact}
\newtheorem{myclaim}{Claim}
\newtheorem{mysenx}{Theorem}

\numberwithin{mydef}{section}

\DeclareMathOperator{\cof}{cof}

\DeclareMathOperator{\supp}{supp}

\DeclareMathOperator{\im}{range}

\DeclareMathOperator{\otp}{otp}
\DeclareMathOperator{\cf}{cf}
\DeclareMathOperator{\SCH}{SCH}

\DeclareMathOperator{\Add}{Add}

\DeclareMathOperator{\Coll}{Coll}

\DeclareMathOperator{\TP}{TP}

\DeclareMathOperator{\On}{On}
\DeclareMathOperator{\SLIP}{SLIP}

\newcommand{\dA}{\mathbb{A}}

\newcommand{\dC}{\mathbb{C}}

\newcommand{\dL}{\mathbb{L}}

\newcommand{\dP}{\mathbb{P}}
\newcommand{\dQ}{\mathbb{Q}}
\newcommand{\dR}{\mathbb{R}}
\newcommand{\dS}{\mathbb{S}}
\newcommand{\dT}{\mathbb{T}}

\newcommand{\uhr}{\upharpoonright}

\newcommand{\ZFC}{\mathsf{ZFC}}
\newcommand{\GCH}{\mathsf{GCH}}

\newcommand{\AP}{\mathsf{AP}}
\newcommand{\LIP}{\mathsf{LIP}}
\newcommand{\CH}{\mathsf{CH}}
\newcommand{\s}{\subseteq}
\newcommand\cat[1]{{}^\curvearrowright #1}
\newcommand\forces{\Vdash}

\title[On Shelah's Approachability Ideal]{On Shelah's Approachability Ideal}
\author[Jakob]{Hannes Jakob}
\address[Jakob]{Department of Mathematics, College of Science, University of North Texas, Denton TX, 76201, USA}
\email{hannes.jakob@unt.edu}
\urladdr{https://hannesjakob.github.io}
\author[Poveda]{Alejandro Poveda}
\address[Poveda]{Department of Mathematics and Center of Mathematical Sciences and Applications, Harvard University, Cambridge MA, 02138, USA}
\email{alejandro@cmsa.fas.harvard.edu}
\urladdr{www.alejandropovedaruzafa.com}


\begin{document}

\begin{abstract}
We solve a long-standing open problem of Shelah regarding the \emph{Approachability Ideal} $I[\kappa^+]$. Given a singular cardinal $\aleph_\gamma$, a regular cardinal $\mu\in (\cf(\gamma),\aleph_\gamma)$ and assuming appropriate large cardinal hypotheses, we cons\-truct a model of $\mathsf{ZFC}$ in which $\aleph_{\gamma+1} \cap \cof(\mu) \notin I[\aleph_{\gamma+1}]$. 
This provides a definitive answer to a question of  Shelah from the 80's. 
In addition, assuming large cardinals, we  construct a model of $\mathsf{ZFC}$ in which 
the approachability property fails, simultaneously, at every singular cardinal. This is a major milestone in the solution of a  question of Foreman and Magidor from the 80's.
\end{abstract}

\maketitle

\section{Introduction}
The study of singular cardinals stands out as one of the most fascinating \-areas of inquiry in modern set theory. It was Georg Cantor who first formalized the concept of the infinite and initiated the investigation of the continuum. As part of his pioneering analysis, Cantor formulated the \emph{Continuum Hypothesis} ($\CH$), which posits that the cardinality of $\mathbb{R}$ (i.e., $2^{\aleph_0}$) is as small as possible—namely, $\aleph_1$, the first uncountable cardinal. The status of $\CH$ within the broader mathematical paradigm remained elusive to many leading figures until Kurt G\"odel \cite{GodelL} showed that $\CH$ cannot be refuted from the standard axioms of mathematics ($\mathsf{ZFC}$), and later Paul Cohen \cite{cohen1,cohen2}, inventing the set-theoretic method of \emph{forcing}, showed that $\CH$ cannot be proved from $\mathsf{ZFC}$ either. These two results established the independence of $\CH$ from the standard foundation of mathematics, setting the tone for sixty years of groundbreaking discoveries in which the phenomenon of independence played a central role. With the subsequent development of forcing, Easton demonstrated that there are virtually no $\mathsf{ZFC}$-provable restrictions on the behavior of the power-set function $\kappa \mapsto 2^\kappa$ \cite{MR269497}. This inaugurated a series of results showing that, at least for (so-called) \emph{regular cardinals}—such as $\aleph_0$ (the cardinality of the integers) or $\aleph_1$—$\mathsf{ZFC}$ alone settles little about cardinal arithmetic. Confounding all expectations, a series of groundbreaking results pioneered by Galvin–Hajnal \cite{MR376359}, Silver \cite{MR0429564}, Magidor \cite{MagidorSingularsII} and Shelah \cite{ShelahBook} revealed that \emph{singular cardinals} (such as $\aleph_\omega$ and $\aleph_{\omega_1}$) obey far deeper constraints, and that $\mathsf{ZFC}$ itself proves remarkable structural theorems. For instance, a celebrated result of Shelah shows that if $2^{\aleph_n}<\aleph_\omega$ holds for all integers $n$ then there is a provable upper bound on $2^{\aleph_\omega}$. This is in stark contrast with power sets of regular cardinals, where no such upper bound can be established in $\mathsf{ZFC}$ alone as per the work of Easton \cite{MR269497}. 

Signs of the striking nature of singular cardinals have been found well beyond the realm of cardinal arithmetic. A paradigmatic example is Shelah’s Singular Compactness Theorem~\cite{ShelahCompactness} in infinite abelian group theory, which shows that every abelian group $G$ that is $\aleph_\omega$-generated and almost free must be free. Magidor--Shelah extensively studied such phenomena in their classical paper~\cite{MagShegroups}, and they are now considered part of the standard landscape of the field. Singular cardinals have also surfaced in functional analysis. Improving a  classical, influential, work by Gowers and Maurey~\cite{GowersMaurey}, Dodos, López-Abad, and Todorcevic~\cite{AbadDodosTodorcevic} showed that, consistently, every Banach space of density character at most $\aleph_\omega$ contains an unconditional basic sequence---a result that is arguably optimal.

\smallskip

The present manuscript continues this  major line of research in modern set theory by analyzing a related core construction: Shelah's approachability ideal, $I[\kappa]$.

\smallskip

 Let $\kappa$ be a regular uncountable cardinal. A set $S\s \kappa$ is called \emph{stationary} if it intersects every \emph{club} set $C\s \kappa$ – that is, it intersects every subset of $\kappa$ that is closed and unbounded in  its ordinal topology. Intuitively speaking, club sets correspond to \emph{measure one} sets whereas stationary sets parallel sets having \emph{positive measure}. Stationarity is a cornerstone concept in infinitary combinatorics which has been critically utilized in Module Theory \cite{EklofMekler,CortesPoveda}, Cohomology \cite{Bergfalk} or Analysis \cite{Farah, FarahBook, MagidorPlebanek} among many other areas in mathematics. In set theory,  the preservation of these objects after passing to generic extensions becomes pivotal in the proof of some central results in the field \cite{MagidorReflectionStat, CumForeMagSquareScalesStatRefl, PartIII, BenNeriaHayutUnger}.

 \smallskip

In his classical paper \cite{ShelahSuccSingCard}, Shelah introduced (if implicitly) the \emph{approachability ideal} $I[\kappa]$ and the corresponding notion of the \emph{approachability property} $\AP_{\kappa}$, asserting that $\kappa^+\in I[\kappa^+]$ (i.e., $I[\kappa^+]$ is improper). Through an incisive analysis,  Shelah bridges  $I[\kappa]$ with the preservation of stationary sets under set-theoretic forcing. Specifically, he shows that  if $S \subseteq \kappa \cap \cof(\lambda)$\footnote{$\kappa\cap \cof(\lambda)$ is the standard notation for the collection of ordinals $\alpha<\kappa$ with cofinality $\lambda$.} is stationary and $S \in I[\kappa]$ then the stationarity of $S$ is preserved by so-called ${<}\lambda^+$-closed posets (see Section~\ref{sec: prelimminaries}). Thereby, under $\AP_\kappa$ any stationary set $S \subseteq \kappa^+ \cap \cof(\lambda)$ is preserved by ${<}\lambda^+$-closed forcing. This is somewhat sharp as if $S \notin I[\kappa]$, then even garden variety ${<}\lambda^+$-closed forcings, such as the \emph{Levy collapse} $\Coll(\lambda, \kappa)$, destroy the stationarity of $S$.

Shelah showed that $I[\kappa]$ is a normal ideal and demonstrated that it is rich enough to contain \emph{many} stationary sets. Namely, if $\lambda<\kappa$ are both regular and $\lambda^+<\kappa$ then $I[\kappa]$ contains a stationary subset of $\kappa\cap \cof(\lambda)$. Moreover, if $\kappa$ itself is a regular cardinal then $\kappa^+\cap \cof({<}\kappa)\in I[\kappa^+]$, and thus $I[\kappa^+]$ is completely determined by membership of stationary subsets of $\kappa^+\cap \cof(\kappa)$. As usual, the situation with successors of singular cardinals is way more subtle – in this case $\ZFC$ can only establish  $\kappa^+\cap \cof({\leq} \cf(\kappa))\in I[\kappa^+]$. Whether or not this result is sharp constitutes a major problem and is one of the issues that is under investigation in this manuscript.

\smallskip

Ever since its conception, investigations of $I[\kappa]$ have constituted a major area of research in set theory. Today's understanding of this subject has been substantially enriched by the work of Shelah~\cite{ShelahApproachability,MR1261217}, Foreman--Magidor~\cite{ForemanMagidorLCCounterexamples,ForemanMagidorVeryWeak}, Mitchell~\cite{MitchellNonstationary}, Gitik–Krueger~\cite{KruegerGitikApproach}, Sharon–Viale~\cite{VialeSharon}, Gitik–Rinot~\cite{RinotGitik}, Unger~\cite{UngerSuccessiveApproach}, Cummings et al.~\cite{CummingsFriedmanMagidorRinotSinapovaEightfold}, Krueger~\cite{KruegerApproachMaximal}, and Mohammadpour--Velickovic~\cite{MohammadpourVelickovic}, among many other authors. Current interest in $I[\kappa]$ extends beyond its original connection with the preservation of stationary sets. This is due to the interplay between $I[\kappa]$ and other central principles in combinatorial set theory, such as Jensen’s $\square_\kappa$-principle and its relatives, and the Tree Property ($\TP$). 

In turn, this connects $I[\kappa]$ with two major questions in the field, which have inspired the present work. Namely,

\begin{myque}[Shelah, 80's]\label{que: Shelah}
Let $\kappa$ be a singular cardinal. Must  the approachability ideal $I[\kappa^+]$ contain a club set relative to cofinality $\cf(\kappa)^{++}$?\footnote{See \cite[Question~3.5]{ForemanSurvey} for a particular instance of this question. The problem is also posed on page 1280 of \cite{EisworthHandbook} where it is deemed a ``major open problem in this area".}
\end{myque}

\begin{myque}[Foreman, Magidor, 80's]\label{que: Magidor}
Must there be a regular cardinal  $\kappa\geq \aleph_2$ for which $ \mathrm{TP}_\kappa$ fails?
\end{myque}

In the first part of the manuscript we provide a definitive, negative, answer to Shelah's Question~\ref{que: Shelah} by proving the following consistency result:

\begin{mysenx}\label{ThmA}
Assume the $\GCH$ holds and that there is a supercompact cardinal. Given a singular cardinal  $\aleph_\gamma$ and a regular cardinal $\mu\in (\cf(\gamma),\aleph_\gamma)$  there is a model of $\ZFC$ where $\aleph_{\gamma+1}\cap \cof(\mu)\notin I[\aleph_{\gamma+1}]$ holds.
\end{mysenx}

Assuming the consistency of certain \emph{large cardinal axiom} called a \emph{supercompact},  Shelah proved in \cite{ShelahSuccSingCard}  the consistency of $\ZFC$ with  
$\aleph_{\omega+1} \cap \cof(\omega_1) \notin I[\aleph_{\omega+1}]$. Shelah's argument was limited to cofinality $\omega_1$ (see page~\pageref{discussion on shelah's theorem} for a discussion)  and so the author asked  whether or not the statement 
$\aleph_{\omega+1} \cap \cof(\omega_2) \notin I[\aleph_{\omega+1}]$
was consistent with $\ZFC$ \cite[Question~8.3]{ForemanSurvey}. This particular instance of Question~\ref{que: Shelah} remained open for quite a long time until it was recently answered in the affirmative by Jakob--Levine in \cite{JakobLevine}. 
Subsequently, the first author extended this result in \cite{JakobTotalFailure}, constructing a model of $\mathsf{ZFC}$ in which, for every singular cardinal $\mu$ of countable cofinality and every regular cardinal $\gamma \in (\aleph_0, \mu)$,
$\mu^+ \cap \cof(\gamma) \notin I[\mu^+]$ holds. In spite of these recent breakthroughs, Shelah's problem  remained unsettled. The reason is that the methods developed in \cite{JakobLevine,JakobTotalFailure} were largely specific to singulars with countable cofinality due to their dependence on \emph{Namba-style  forcings}. As a result, new and more sophisticated technologies must be developed. 
Among the novelties offered by this paper there is the construction of a poset $\dP(\vec{\kappa}, \vec{\mathcal{I}}, \vec{\mathcal{B}})$ which plays an analogue role to the Namba-like forcing from \cite[\S4]{JakobLevine}, yet with the advantage that it can dispose with uncountable cofinalities. This poset is inspired by former constructions stemming from the theory of Prikry-type forcings \cite{GitikHandbook}. In the vernacular language of this field, $\dP(\vec{\kappa}, \vec{\mathcal{I}}, \vec{\mathcal{B}})$ is an \emph{ideal-based}  version of a Magidor-product of \emph{one-point Prikry forcings}. 
This poset has the crucial property that, in a model obtained via the two-step iteration $\dP(\vec{\kappa}, \vec{\mathcal{I}}, \vec{\mathcal{B}}) * \dot{\Coll}(\mu, (\sup \vec{\kappa})^+)$, every cofinal function from $\mu$ into $(\sup\vec\kappa)^+$ has an initial segment that is {not covered by any ground-model set of size ${<} \sup \vec{\kappa}$}. This is the key to show that in the corresponding generic extension the stationary set $(\sup\vec\kappa)^+\cap \cof(\mu)$ is non-approachable.  Another technical novelty of the forcing $\dP(\vec{\kappa}, \vec{\mathcal{I}}, \vec{\mathcal{B}})$ is that its definition  hinges on ideals rather than ultrafilters. This makes the standard \emph{Prikry analysis} of the poset ostensibly more subtle than usual.

\smallskip

One of the key features of Prikry-type forcings is their amenability to a  theory of forcing iterations \cite{MagidorIdentityCrises, GitikHandbook}. After identifying $\dP(\vec{\kappa}, \vec{\mathcal{I}}, \vec{\mathcal{B}})$ as a member of this family of posets
 we proceed to iterate the construction used in Theorem~\ref{ThmA} to obtain the following global result:

\begin{mysenx}\label{ThmD}
    Assume the $\GCH$ and that there is a proper class of supercompact cardinals. Then there is a model of $\ZFC$ where $\AP_{\kappa}$ fails for every singular cardinal.

    Moreover, for each singular cardinal $\kappa$, there are unboundedly many regular cardinals $\delta<\kappa$ for which $\kappa^+\cap\cof(\delta)\notin I[\kappa^+]$.
\end{mysenx}

Theorem~\ref{ThmD} is tied to Foreman--Magidor's Question~\ref{que: Magidor} via Jensen's \emph{weak square principle} $\square^*_\kappa$. A classical result of Jensen \cite{Jensen} states that $\square^*_\kappa$ is equivalent to the existence of a \emph{special} $\kappa^+$-Aronszajn tree; that is, a $\kappa^+$-tree $(T, <_T)$ admitting a function $f \colon T \rightarrow \kappa$ such that $x <_T y$ implies $f(x) \neq f(y)$. The key feature of such trees is that they remain $\kappa^+$-Aronszajn (i.e., branchless) in any outer model of $\mathsf{ZFC}$ where $\kappa^+$ remains a cardinal. Consequently, special $\kappa^+$-Aronszajn trees provide strong failures of $\mathrm{TP}_{\kappa^+}$. Since $\square^*_\kappa$ implies $\AP_\kappa$, Theorem~\ref{ThmD} yields the first model of $\ZFC$ in which no special $\kappa^+$-Aronszajn trees exist for any singular cardinal $\kappa$.

It is widely recognized among specialists that getting the Tree Property at the successor of every  singular cardinal is one of the principal technical challenges in resolving Foreman–Magidor Question~\ref{que: Magidor} (see e.g., \cite{MagSheTree, NeemanTree, NeemanUpToAleph, CummingsetalTree, CummingsetalTreeII}. The methods developed here thus represent a substantial advance toward the resolution of this second long-standing problem.

\smallskip

The paper is organized as follows. In Section~\ref{sec: prelimminaries} we present the reader with preliminary definitions and results concerning colorings on successors of singular cardinals and basic forcing facts for later use.  In Section~\ref{sec: Higher Namba} we define the poset $\dP(\vec{\kappa}, \vec{\mathcal{I}}, \vec{\mathcal{B}})$ and establish its key properties. After that we resolve Shelah's problem by proving Theorem~\ref{ThmA}. As a bonus result, in this section we use our methods to improve the main result of Gitik–Sharon \cite{GitikSharon}. In Section~\ref{sec: failure everywhere} we construct the iteration that yields Theorem~\ref{ThmD}.  The paper is concluded with a section with open questions and concluding remarks. The manuscript is self-contained, and the notations herein follow the standard vernacular set theory. Only basic acquaintance with  the theory of forcing and  large cardinals is presumed.

\section{Preliminaries}\label{sec: prelimminaries}

\subsection{Notations and conventions}
Our forcing convention is that $p\leq q$ means that $p$ extends $q$. Given cardinals $\kappa<\lambda$, $\Coll(\kappa,\lambda)$ (resp. $\Coll(\kappa,{<}\lambda)$) denotes the usual Levy collapse collapsing $\lambda$ (resp. every cardinal in $(\kappa,\lambda)$) to $\kappa$.  $\Add(\kappa,\lambda)$ denotes Cohen forcing adjoing $\lambda$-many subsets to $\kappa$.  For regular cardinals $\lambda<\kappa$ we denote by either $E^\kappa_\lambda$ or $\kappa\cap \cof(\lambda)$ the set of all $\alpha<\kappa$ with cofinality $\lambda$. For a regular cardinal $\Theta$, $H(\Theta)$ denotes the collection of sets of hereditary cardinality less than $\Theta$. Given a function $f$ and a set $A$, $f[A]$ denotes $\{f(a)\mid a\in A\}.$ Given an ideal $\mathcal{I}$ on a (non-empty) set $X$ we denote by $\mathcal{I}^+$ the collection of \emph{$\mathcal{I}$-positive} sets; namely, $\mathcal{I}^+:=\{A\subseteq \mathcal{P}(X)\mid A\notin \mathcal{I}\}.$ Given an uncountable regular cardinal $\kappa$, we say that $\mathcal{I}$ is ${<}\kappa$-complete if $\mathcal{I}$ is closed under ${<}\kappa$-sized unions of its members. 

Some of the arguments in this paper will require using  \emph{large cardinals}. The standard reference in this matter is Kanamori's text \cite{KanamoriHigherInfinite} where the reader is referred for a more comprehensive exposition. For the sake of completeness, we remind the definition of two key large cardinal notions:  An uncountable cardinal $\kappa$ is \emph{measurable} if there is a non-pricipal $\kappa$-complete ultrafilter on $\kappa$. We say that $\kappa$ is  \emph{supercompact} if for each  $\lambda\geq \kappa$ there is an elementary embedding $j\colon V\rightarrow M$ with \emph{critical point} $\kappa$, $j(\kappa)>\lambda$ and $M$  closed under $\lambda$-sequences of its members.

\subsection{Forcing}
In this section we provide a few preliminaries on the method of set-theoretic forcing. For the general theory of forcing we refer our readers to \cite{Kunen}.

\begin{mydef}
    Let $(\dP,\leq_{\dP})$ and $(\dQ,\leq_{\dQ})$ be posets. A function $\pi\colon\dP\to\dQ$ is a \emph{projection} if the following holds:
    \begin{enumerate}
        \item $\pi(1_{\dP})=1_{\dQ}$.
        \item If $p'\leq_{\dP}p$, then $\pi(p')\leq_{\dQ}\pi(p)$.
        \item Whenever $q\leq_{\dQ}\pi(p)$, there is $p'\leq_{\dP}p$ such that $\pi(p')\leq_{\dQ}q$.
    \end{enumerate}
\end{mydef}

If there exists a projection from $\dP$ to $\dQ$, any forcing extension by $\dQ$ can be forcing extended to an extension by $\dP$. This is made precise via the next definition:

\begin{mydef}
    Let $\dP$ and $\dQ$ be posets and $\pi\colon\dP\to\dQ$ a projection. Let $G$ be $\dQ$-generic. In $V[G]$, let $\dP/G$ consist of all those elements $p\in\dP$ with $\pi(p)\in G$, ordered as a suborder of $\dP$. We let $\dP/\dQ$ be a $\dQ$-name for $\dP/\dot{G}$ and call $\dP/\dQ$ the \emph{quotient forcing} of $\dP$ and $\dQ$.
\end{mydef}

\begin{myfact}
    Let $\dP$ and $\dQ$ be poset and $\pi\colon\dP\to\dQ$ a projection. Let $G$ be $\dQ$-generic over $V$ and let $H$ be $\dP/G$-generic over $V[G]$. Then $H$ is $\dP$-generic over $V$ and $G=\pi[H]$. In particular, $V[G][H]=V[H]$.\qed
\end{myfact}

Projections onto iterations are often obtained by considering the so-called \emph{termspace forcing}, an idea due to Laver and subsequently developed by Foreman \cite{ForemanSatIdeal}:

\begin{mydef}[Laver]\label{def: Termspace}
    Let $(\dP,\leq)$ be a poset and let $(\dot{\dQ},\dot{\leq})$ be a $\dP$-name for a poset. The poset $\dT(\dP,\dot{\dQ})$ consists of all $\dP$-names for elements of $\dot{\dQ}$, ordered by $\dot{q}'\prec\dot{q}$ if and only if $1_{\dP}\Vdash\dot{q}'\dot{\leq}\dot{q}$.
\end{mydef}

Using standard arguments on names, one shows:

\begin{mylem}
    Let $\dP$ be a poset and $\dot{\dQ}$ a $\dP$-name for a poset. The identity function is a projection from $\dP\times\dT(\dP,\dot{\dQ})$ onto $\dP*\dot{\dQ}$.\qed
\end{mylem}

Recall that a poset $\dP$ is \emph{separative} if whenever $p_0\not\leq p$, there is $p_1\leq p$ such that $p_0$ and $p_1$ are incompatible. We also record the following easy lemma for later:

\begin{mylem}\label{lemma: Termspace separative}
    Let $\dP$ be a poset and $\dot{\dQ}$ a $\dP$-name for a poset. If $\dP$ forces that $\dot{\dQ}$ is separative, then $\dT(\dP,\dot{\dQ})$ is separative.
\end{mylem}

\begin{proof}
    Let $\dot{q}_0,\dot{q}\in\dT(\dP,\dot{\dQ})$ be such that $\dot{q}_0\not\prec\dot{q}$, i.e. $1_{\dP}\not\Vdash\dot{q}_0\dot{\leq}\dot{q}$. Ergo there is $p\in\dP$ such that $p\Vdash\dot{q}\not\dot{\leq}\dot{q}'$. Since $\dot{\dQ}$ is forced to be separative, by the maximum principle we can find $\dot{q}'$ such that $p\Vdash(\dot{q}'\dot{\leq}\dot{q})\wedge(\dot{q}'\perp\dot{q}_0)$. Now let $\dot{q}_1$ be such that $p\Vdash\dot{q}'=\dot{q}_1$ and conditions incompatible with $p$ force $\dot{q}'=\dot{q}_1$. Then $\dot{q}_1\prec\dot{q}$ and clearly there is no $\dot{q}''$ such that $1_{\dP}\Vdash\dot{q}''\dot{\leq}\dot{q}_0,\dot{q}_1$. In particular, $\dot{q}_0$ and $\dot{q}_1$ are $\prec$-incompatible.
\end{proof}

\begin{mydef}
    Let $\kappa$ be an uncountable regular cardinal and $\mathbb{P}$ be a poset. We say that $\mathbb{P}$ is \emph{${<}\kappa$-closed} if every $\leq$-decreasing sequence $\langle p_\alpha\mid \alpha<\theta\rangle$ of conditions in $\mathbb{P}$ with $\theta<\kappa$ admits a $\leq$-lower bound. We say that $\mathbb{P}$ is ${<}\kappa$-directed closed if every directed subset $D\s \mathbb{P}$ of size ${<}\kappa$ has a lower bound.\footnote{Recall that $D\s \mathbb{P}$ is called directed if any two conditions $p,q$ are compatible via a member of $D$.}
\end{mydef}

We will later use posets which do not have lower bounds for arbitrary sequences of a given length but only those which were constructed according to a nice strategy. 

\begin{mydef}\label{def: INC COM game}
    Let $\dP$ be a poset and $\gamma$ an ordinal. The \emph{completenss game} $\mathcal{G}(\dP,\gamma)$ is defined as follows: The game lasts $\gamma$ many rounds. COM starts by playing $1_{\dP}$. At any stage $\delta$, we have a position $(p_{\alpha},q_{\alpha})_{\alpha<\delta}$. Then COM  has to play a condition $p_{\delta}$ which is a lower bound of $(q_{\alpha})_{\alpha<\delta}$ and INC has to play a condition $q_{\delta}\leq p_{\delta}$. COM wins if they can continue to play for all $\gamma$ many rounds. Otherwise, INC wins.

    For a cardinal $\kappa$, we say that $\dP$ is \emph{$\kappa$-strategically closed} if COM has a winning strategy in $\mathcal{G}(\dP,\kappa)$. We say that $\dP$ is \emph{${<}\,\kappa$-strategically closed} if COM has a winning strategy in $\mathcal{G}(\dP,\gamma)$ for every $\gamma<\kappa$.
\end{mydef}

Foreman was able to relate the distributivity of a poset $\dP$ to the nonexistence of a winning strategy for INC in the completeness game. This characterization often proves very useful because it allows one to use distributivity assumptions to construct sequences $(p_{\alpha})_{\alpha<\delta}$ where $p_{\alpha}$ depends on the history $(p_{\beta})_{\beta<\alpha}$, knowing that this construction will at least succeed once:

\begin{mysen}[Foreman \cite{ForemanBool}]\label{Theorem: Foreman Distributivity}
    Suppose that $\dP$ is a poset and $\kappa$ is a cardinal. Then $\dP$ is ${<}\,\kappa^+$-distributive if and only if $\mathrm{INC}$ does not have a winning strategy in $\mathcal{G}(\dP,\kappa+1)$.
\end{mysen}

A critical fact connecting all the notions defined so far is due to Easton:

\begin{myfact}[Easton's Lemma]\label{lemma: Eastons lemma}
Let $\kappa$ be an uncountable regular cardinal. Assume that  $\mathbb{P}$ is $\kappa$-cc and $\mathbb{Q}$ is ${<}\kappa$-closed. Then: 
\begin{enumerate}
    \item $\forces_{\mathbb{P}\times \mathbb{Q}}``\kappa$ is a regular uncountable cardinal".
    \item $\forces_{\mathbb{P}}``\check{\mathbb{Q}}$ is ${<}\kappa$-distributive".
     \item $\forces_{\mathbb{Q}}``\check{\mathbb{P}}$ is $\kappa$-cc".\qed
\end{enumerate}
\end{myfact}
Our posets will derive their regularity properties from being of \emph{Prikry-type}.

\begin{mydef}[Gitik]
    Let $(\dP,\leq,\leq_0)$ be such that $(\dP,\leq)$ and $(\dP,\leq_0)$ are posets and $\leq$ refines $\leq_0$. We say that \emph{$(\dP,\leq,\leq_0)$ is a Prikry-type forcing} if whenever $p\in\dP$ and $\sigma$ is a sentence in the forcing language, there exists $q\leq_0p$ which decides $\sigma$.
\end{mydef}

Equivalently, $(\dP,\leq,\leq_0)$ is a Prikry-type forcing if and only if names for elements of the ordinal $2$ can be decided using $\leq_0$- (i.e. \emph{pure} or \emph{direct}) extensions. This idea leads to the following definition:

\begin{mydef}\label{def: pure decidability}
    Let $(\dP,\leq,\leq_0)$ be a Prikry-type forcing and $\mu$ a cardinal.
    \begin{enumerate}
        \item We say that $(\dP,\leq,\leq_0)$ has \emph{pure $\mu$-decidability} if whenever $\tau$ is a $\dP$-name and $p\in\dP$ forces $\tau<\check{\mu}$, there is $\alpha<\mu$ and $q\leq_0p$ such that $q\Vdash\tau=\check{\alpha}$.
        \item We say that $(\dP,\leq,\leq_0)$ has \emph{almost pure $\mu$-decidability} if whenever $\tau$ is a $\dP$-name and $p\in\dP$ forces $\tau<\check{\mu}$, there is $\alpha<\mu$ and $q\leq_0p$ such that $q\Vdash\tau<\check{\alpha}$.
    \end{enumerate}
\end{mydef}

We now turn to the concept of iterating Prikry-type forcings. The most important point here is that we want retain that any such iteration still has the Prikry property. A notion of iteration which satisfies this property was invented by Magidor in  \cite{MagidorIdentityCrises}. These iterations can be defined with Easton support, non-stationary and full support, but we will only use Easton-support iterations.

\begin{mydef}\label{def: Easton support Magidor}
	Let $((\dP_{\alpha},\leq_{\alpha},\leq_{\alpha,0}),(\dot{\dQ}_{\alpha},\dot{\leq}_{\alpha},\dot{\leq}_{\alpha,0}))_{\alpha<\rho}$ be a sequence such that each $(\dP_{\alpha},\leq_{\alpha},\leq_{\alpha,0})$ is a poset and each $(\dot{\dQ}_{\alpha},\dot{\leq}_{\alpha},\dot{\leq}_{\alpha,0})$ is a $\dP_{\alpha}$-name for a Prikry-type poset. We define the statement ``$((\dP_{\alpha},\leq_{\alpha},\leq_{\alpha,0}),(\dot{\dQ}_{\alpha},\dot{\leq}_{\alpha},\dot{\leq}_{\alpha,0}))_{\alpha<\rho}$ is an Easton-support Magidor iteration of Prikry-type forcings of length $\rho$'' by induction on $\rho$.
		
	$((\dP_{\alpha},\leq_{\alpha},\leq_{\alpha,0}),(\dot{\dQ}_{\alpha},\dot{\leq}_{\alpha},\dot{\leq}_{\alpha,0}))_{\alpha<\rho}$ is an Easton-support Magidor iteration of Prikry-type forcings of length $\rho$ if $((\dP_{\alpha},\leq_{\alpha},\leq_{\alpha,0}),(\dot{\dQ}_{\alpha},\dot{\leq}_{\alpha},\dot{\leq}_{\alpha,0}))_{\alpha<\rho'}$ is an Easton-support Magidor iteration of Prikry-type forcings of length $\rho'$ for every $\rho'<\rho$ and moreover:
	\begin{enumerate}
		\item If $\rho=\rho'+1$, then $(\dP_{\rho},\leq_{\rho}):=(\dP_{\rho'},\leq_{\rho'})*(\dot{\dQ}_{\rho'},\dot{\leq}_{\rho'})$ and $(p',\dot{q}')\leq_{\rho,0}(p,\dot{q})$ if and only if $p'\leq_{\rho',0}p$ and $p'\Vdash\dot{q}'\dot{\leq}_{\rho',0}\dot{q}$.
		\item If $\rho$ is a limit, then $\dP_{\rho}$ consists of all functions $p$ on $\rho$ such that
		\begin{enumerate}
			\item For all $\alpha<\rho$, $p\uhr\alpha\in\dP_{\alpha}$,
			\item If $\rho$ is inaccessible and $|\dP_{\alpha}|<\rho$ for every $\alpha<\rho$, then there is some $\beta<\rho$ such that for all $\gamma\in(\beta,\rho)$, $p\uhr\gamma\Vdash p(\gamma)=1_{\dot{\dQ}_{\gamma}}$.
		\end{enumerate}
		and the following holds:
		\begin{enumerate}
			\item[(i)] $p'\leq_{\rho}p$ if and only if $p'\uhr\rho'\leq_{\rho'}p\uhr\rho'$ for every $\rho'<\rho$ and there exists a finite subset $b$ such that whenever $\rho'\notin b$ and $p\uhr\rho'\not\Vdash p(\rho')=1_{\dot{\dQ}_{\rho'}}$, 
            then $p'\uhr\rho'\Vdash p'(\rho')\dot{\leq}_{\rho',0} p(\rho')$.
			\item[(ii)] $p'\leq_{\rho,0}p$ if and only if $p'\leq_{\rho}p$ and the set $b$ is empty.
		\end{enumerate} 
	\end{enumerate}
\end{mydef}
	
We have the following (see \cite[Section 6.3]{GitikHandbook} or \cite[Lemma 2.4]{JakobLevine}):
	
\begin{mylem}\label{lemma: Easton support Magidor}
	Let $((\dP_{\alpha},\leq_{\alpha},\leq_{\alpha,0}),(\dot{\dQ}_{\alpha},\dot{\leq}_{\alpha},\dot{\leq}_{\alpha,0}))_{\alpha<\rho}$ be an Easton support Magidor iteration of Prikry-type forcings of length $\rho$. Let $\mu$ be a cardinal.
    
    Assume that for each $\alpha<\rho$, $\dP_{\alpha}$ forces that $\dot{\dQ}_{\alpha}$ has (almost) pure $\check{\mu}$-decidability. Then for each $\alpha<\rho$, $\dP_{\alpha}$ has (almost) pure $\check{\mu}$-decidability. \qed
\end{mylem}

Prikry-type forcings mostly derive their regularity properties from the Prikry property together with the closure of their pure extension ordering. In iterations, notice that the pure extension ordering on the iteration is not the same as the iteration of the pure extension orderings. Despite this fact, since the regular ordering refines the pure extension ordering, one checks easily that the usual proof of the iterability of closure properties goes through:

\begin{mylem}\label{lemma: Iteration pure closure}
    Let $\mu$ be a cardinal. Let $((\dP_{\alpha},\leq_{\alpha},\leq_{\alpha,0}),(\dot{\dQ}_{\alpha},\dot{\leq}_{\alpha},\dot{\leq}_{\alpha,0}))_{\alpha<\rho}$ be an Easton support Magidor iteration of Prikry-type forcings of length $\rho$. If $\mu$ is below the first inaccessible cardinal and for each $\alpha<\rho$, $\dP_{\alpha}$ forces that $\dot{\dQ}_{\alpha}$ is ${<}\,\check{\mu}$-directed closed, then each $\dP_{\alpha}$ is ${<}\,\mu$-directed closed.\qed
\end{mylem}

Lastly, we show the existence of a nice projection for Prikry-type forcings (similar to an idea of Foreman, see \cite{ForemanSatIdeal}).

\begin{mylem}\label{lemma: Projection Easton support Magidor}
    Let $((\dP_{\alpha},\leq_{\alpha},\leq_{\alpha,0}),(\dot{\dQ}_{\alpha},\dot{\leq}_{\alpha},\dot{\leq}_{\alpha,0}))_{\alpha<\rho}$ be an Easton support Magidor iteration of Prikry-type forcings of length $\rho$. Let $\dT:=\prod_{\alpha<\kappa}\dT((\dP_{\alpha},\leq_{\alpha}),(\dot{\dQ}_{\alpha},\dot{\leq}_{\alpha,0}))$, where the product is taken with Easton support. Then the identity is a projection from $\dT$ onto the direct extension ordering on the Easton limit of the iteration $((\dP_{\alpha},\leq_{\alpha},\leq_{\alpha,0}),(\dot{\dQ}_{\alpha},\dot{\leq}_{\alpha},\dot{\leq}_{\alpha,0}))_{\alpha<\rho}$. \qed
\end{mylem}

\begin{proof}
		Let $\dP$ be the Easton-support Magidor limit of $((\dP_{\alpha},\leq_{\alpha},\leq_{\alpha,0}),(\dot{\dQ}_{\alpha},\dot{\leq}_{\alpha},\dot{\leq}_{\alpha,0}))$. Recall that $\dP$ consists of all functions $p$ on $\kappa$ such that $p\uhr\alpha\in\dP_{\alpha}$ for all $\alpha<\kappa$ and the support of $p$ is an Easton set (i.e. it is bounded in every regular cardinal), directly ordered by $p\leq_0 q$ if and only if $p\uhr\alpha\leq_{\alpha,0}q\uhr\alpha$ for all $\alpha<\kappa$. Denote by $\leq_{\dT}$ the ordering on $\dT$.
		
		It is clear that the identity is a function from $\dT$ onto $\dP$ and preserves $\leq_0$: Let $t_0,t_1\in\dT$ with $t_0\leq_{\dT}t_1$. Then $t_0$ is a function on $\kappa$ such that for each $\alpha<\kappa$, $t_0(\alpha)$ is a $(\dP_{\alpha},\leq_{\alpha})$-name such that $1_{\dP_{\alpha}}\Vdash t_0(\alpha)\in\dot{\dQ}_{\alpha}$. It follows by induction on $\alpha$ that $t_0\uhr\alpha\in\dP_{\alpha}$ for every $\alpha<\kappa$ since we use the same support in both cases. For any $\alpha<\kappa$, $1_{\dP_{\alpha}}\Vdash t_0(\alpha)\leq_0t_1(\alpha)$. Ergo $t_0\uhr\alpha\Vdash_{(\dP_{\alpha},\leq_{\alpha})} t_0(\alpha)\dot{\leq}_{\alpha,0}t_1(\alpha)$. It follows again by induction that $t_0\uhr\alpha\leq_{\alpha,0}t_1\uhr\alpha$ for every $\alpha<\kappa$.
		
		Now we show that the identity is a projection. Let $p\in\dP$ and $t\leq_0p$. We want to find $p'\leq_{\dT}p$ with $p'\leq_0t$. By induction on $\alpha$, let $p'(\alpha)$ be a $(\dP_{\alpha},\leq_{\alpha})$-name for an element of $\dot{\dQ}_{\alpha}$ forced by $p'\uhr\alpha$ to be equal to $t(\alpha)$ and by conditions incompatible with $p'\uhr\alpha$ to be equal to $p(\alpha)$. Then the support of $p'$ is equal to the support of $t$ and thus an Easton set. It follows by induction that $p'\uhr\alpha\leq_{\dT}p\uhr\alpha$ and $p'\uhr\alpha\leq_{\alpha,0}t\uhr\alpha$: For any $\alpha$, by induction, $p'(\alpha)$ is in any case forced to be below $p(\alpha)$: $p'\uhr\alpha$ is in particular $\leq$-below $t\uhr\alpha$ and thus forces $p'(\alpha)\leq_0p(\alpha)$ by assumption and conditions incompatible with $p'\uhr\alpha$ outright force $p'(\alpha)=p(\alpha)$. It also follows that $p'\uhr\alpha\Vdash p'(\alpha)=t(\alpha)$ and thus $p'\uhr\alpha+1\leq_{\alpha+1,0} t\uhr\alpha+1$.
	\end{proof}

\subsection{Approachability}

In this section we garner  for later use some standard notations and facts around Shelah's approachability ideal. For a more comprehensive account on the matter we refer to Eisworth's excellent handbook chapter \cite{EisworthHandbook}.

\begin{mydef}
    Let $\kappa$ be a regular cardinal and $\bar{a}=\langle a_\alpha\mid \alpha<\kappa\rangle$ a sequence of bounded subsets of $\kappa$. A limit ordinal $\alpha<\kappa$ is called \emph{approachable with respect to $\bar{a}$} if there is an unbounded set $A\s \alpha$ with $\otp(A)=\cf(\alpha)$ such that $$\{A\cap \beta\mid \beta<\alpha\}\s \{a_\beta\mid \beta<\alpha\}.$$ 
\end{mydef}

\begin{mydef}[The approachability ideal]
Let $\kappa$ be a regular cardinal. A set $S\s \kappa$ is in $I[\kappa]$ if there is a sequence $\bar{a}$ of bounded subsets of $\kappa$ and a club $C\s \kappa$ such that every $\alpha\in C\cap S$ is approachable with respect to $\bar{a}$.
\end{mydef}

The terminology \emph{approachability ideal} is justified by the following fact:

\begin{myfact}[Shelah]
     $I[\kappa]$ is a (possibly improper) normal ideal on $\kappa$. \qed
\end{myfact}
It is not necessarily the case that $I[\kappa]$ is a proper ideal. In fact, in this paper we will construct a model where $I[\kappa^+]$ is improper for all singular cardinals $\kappa$. The assertion that $I[\kappa]$ is an improper ideal is known as the \emph{Approachability Property}:

\begin{mydef}[Approachability]
Let $\kappa$ be a (possibly singular) cardinal. The \emph{approachability property} holds at $\kappa$ if $\kappa^+\in I[\kappa^+]$. We denote this by $\AP_\kappa.$
\end{mydef}
Other standard facts about the ideal $I[\kappa]$ and the approachability property that we may not use in the paper but collect  for the reader's benefit  are:
\begin{myfact}[Shelah]
Let $\kappa$ be a cardinal. Then, the following hold:
    \begin{enumerate}
        \item For every two regular cardinals $\lambda<\kappa$ with $\lambda^+<\kappa$ there is a stationary set $S\s \kappa\cap \cof(\lambda)$ in $I[\kappa]$. \emph{(\cite[Theorem~9.2]{CumNotes})}
        \item $\kappa^+\cap \cof({<}\kappa)\in I[\kappa^+]$ provided $\kappa$ is regular. \emph{(\cite[Lemma 4.4]{ShelahApproachability})}
        \item $\kappa^+\cap \cof({\leq}\cf(\kappa))\in I[\kappa^+]$ provided $\kappa$ is singular. \emph{(\cite[Corollary~3.29]{EisworthHandbook})}
        \item $\square^*_\kappa$ implies $\AP_\kappa$. \emph{(\cite[p.262]{CumNotes})}
        \item If $S\s \kappa\cap\cof(\mu)$ is a stationary set and $\mathbb{P}$  a ${<}\mu^+$-closed forcing then $\forces_{\mathbb{P}}\text{$``\check{S}$ is stationary"}. $ \emph{(\cite[Theorem~20]{ShelahSuccSingCard})} \qed
    \end{enumerate} 
\end{myfact}
If $\kappa$ is regular then  (2) implies that $I[\kappa^+]$ is completely determined by membership to it of stationary subsets of the \emph{critical cofinality} (i.e., $\kappa$). In contrast, if $\kappa$ is singular then there is a wide array of possible configurations for $I[\kappa^+]$, as only $\kappa^+\cap \cof({\leq}\cf(\kappa))\in I[\kappa^+]$  is provable.  In Theorem~\ref{ThmA} we  show that this is  the case.

\smallskip

Even though the approachability ideal $I[\kappa]$ is always \emph{large}, Shelah showed that the existence of supercompact cardinals imposes certain non-trivial restrictions. Namely,
\begin{myfact}[Shelah, \cite{ShelahSuccSingCard}]\label{Fact: supercompact leads to neg AP}
    Suppose that $\kappa$ is a supercompact cardinal and $\lambda>\kappa$ is a singular cardinal with $\cf(\lambda)<\kappa$. Then, there is a singular cardinal $\theta<\kappa$ of cofinality $\cf(\lambda)$ and a stationary set $S\s \lambda^+\cap \cof(\theta^+)$ that is not in $I[\lambda^+].$ In particular, $\AP_\lambda$ fails. $\qed$
\end{myfact}
\label{discussion on shelah's theorem}In this paper we are mostly preoccupied with  the failure of $\AP_\lambda$ at singular cardinals $\lambda$. The standard way to produce the failure of the approachability property at a \emph{down-to-earth} cardinal (such as $\aleph_\omega$) traces back to Shelah \cite[Conclusion~29]{ShelahSuccSingCard}. The basic idea is to start with a supercompact cardinal $\kappa$, let $S\s \kappa^{+\omega+1}\cap \cof(\theta^+)$ a stationary set exemplifying the failure of $\AP_{\kappa^{+\omega}}$ and then ``bring it down" to $\aleph_\omega$. Specifically,  forcing with $\Coll(\aleph_0,\theta)\ast \dot{\Coll}(\aleph_1,{<}\kappa)$ produces a generic extension where the original stationary set $S$ remains stationary, non-approachable and $S\s \aleph_{\omega+1}\cap \cof(\omega_1).$ 
However, the property $\aleph_{\omega+1}\cap \cof(\omega_1)\notin I[\aleph_{\omega+1}]$ was only known to be true in the generic extension if $\theta^+$ was preserved as a cardinal. This posed problems regarding the possible cofinalities at which $\AP$ could fail at small cardinals and also in iterating such a construction to achieve the failure of $\AP$ at many cardinals simultaneously. Due to this, most research surrounding the failure of $\AP_{\lambda}$ for singular cardinals $\lambda$ was focused on achieving this on a club of cardinals (see e.g. the works of Ben-Neria et al. \cite{BenNeriaLambieHansonUngerDiagSCRadin} and Gitik \cite{GitikMethodNotAPnotSCH}).

\smallskip

 A standard way to force the failure of $\AP$ at small singular cardinals is based on a characterization of approachability via colorings, an idea that stems from \cite{ShelahSuccSingCard}.
\begin{mydef}[Colorings]
Let $\kappa$ be  singular  and $d\colon [\kappa^+]^2\rightarrow \cof(\kappa)$ a coloring.

\begin{enumerate}
    \item We say that $d$ is \emph{subadditive} if for each $\alpha<\beta<\gamma<\kappa^+$,
    $$d(\alpha,\gamma)\leq \max\{d(\alpha,\beta),d(\beta,\gamma)\}.$$
    \item We say that $d$ is \emph{normal} if for each  $i<\cf(\kappa)$,
    $$\sup_{\alpha<\kappa^+}|\{\beta<\alpha\mid d(\beta,\alpha)\leq i\}|<\kappa.$$
    \item A point $\alpha<\kappa^+$ is  \emph{$d$-approachable} if there is an unbounded $A\s \alpha$  with
    $$\sup\{d(\beta,\alpha)\mid \beta\in A\}<\cf(\kappa).$$
\end{enumerate}
The collection of all $d$-approachable points will be denoted by $S(d).$
\end{mydef}
A useful characterization of $d$-approachable points that we will use is:
\begin{myfact}[Shelah]\label{fact: characterizing dapproachability}
    An ordinal $\alpha<\kappa^+$ is in $S(d)$ if either $\cof(\alpha)\leq \cof(\kappa)$ or there is $A\s\alpha$ unbounded and $i<\cof(\kappa)$ such that for each $\beta<\gamma$ in $A$, $d(\beta,\gamma)\leq i$.\qed
\end{myfact}

Additionally, we will be using the following statement due to Shelah which vastly simplifies proofs of the preservation of non-approachability.

\begin{myfact}[{\cite[Remark 28]{ShelahSuccSingCard}}]\label{fact: Approachability refinement}
    Suppose $\alpha<\kappa^+$ is $d$-approachable. Whenever $A\subseteq\alpha$ is unbounded, there is $B\subseteq A$ unbounded such that $d\restriction [B]^2$ is bounded.\qed
\end{myfact}

It turns out that if $\kappa$ is a singular cardinal and the GCH holds (or simply, if $\kappa$ is strong limit) the set of $d$-approachable points together with $\kappa^+\cap \cof(\cf({\leq}\kappa)))$ generate the approachability ideal $I[\kappa^+]$ modulo non-stationary sets:

\begin{myfact}[{\cite[Corollary~3.35]{EisworthHandbook}}]\label{fact: eisworth handbook}
    Suppose that $\kappa$ is a strong limit cardinal. Then, $S(d)\cup (\kappa^+\cap \cof(\cf({\leq}\kappa)))$ generates $I[\kappa^+]$ modulo non-stationary sets. \qed
\end{myfact}

\section{The solution to Shelah's problem}\label{sec: Higher Namba}
In this section, we define a Prikry-type forcing that ensures the failure of $\AP_\kappa$ at a singular cardinal $\kappa$ of any prescribed cofinality. The forcing notion will be a two-step iteration consisting of a Magidor-style product $\mathbb{P}(\vec{\kappa}, \vec{\mathcal{I}}, \vec{\mathcal{B}})$ followed by a Levy collapse of the form $\Coll(\mu, (\sup \vec{\kappa})^+)$. 
The poset $\mathbb{P}(\vec{\kappa}, \vec{\mathcal{I}}, \vec{\mathcal{B}})$ is an ideal-based version of a Magidor product of \emph{one-point Prikry forcings}. To guarantee that this poset satisfies both the Prikry and Strong Prikry properties and works as expected we will need to impose certain restrictions on the input parameters  $(\vec{\kappa}, \vec{\mathcal{I}}, \vec{\mathcal{B}})$. These restrictions are captured by an enhancement of the \emph{Laver Ideal Property}.

\begin{mydef}
    Let $\mu<\nu$ be regular cardinals. We let $\LIP(\mu,\nu)$ state that there exists $\mathcal{I}$, a ${<}\,\nu$-complete and normal ideal over $\nu$ such that there is a set $\mathcal{B}\subseteq \mathcal{I}^+$ which is dense in $\mathcal{I}^+$ with respect to $\s$ and $(\mathcal{B},\s)$ is ${<}\,\mu$-directed closed.\footnote{The difference compared to the version of LIP considered in \cite[Definition~3.1]{JakobLevine} is that here we require $B$ to be ${<}\mu$-\textbf{directed} closed.}
\end{mydef}

A list of properties that suffice for a   poset to force instances of $\LIP$ are:
\begin{mydef}
    Let $\mu<\nu$ be regular cardinals with $\nu$ measurable. Let $\dP$ be a poset. We say that $\dP$ is a \emph{$\LIP(\mu,\nu)$-forcing} if the following hold:
    \begin{enumerate}
        \item $\dP$ is separative, ${<}\,\mu$-directed closed, $\nu$-cc and is a subset of $V_\nu$.
        \item There is a club $C\subseteq\nu$ such that whenever $\alpha\in C$ is inaccessible, there is a projection $\pi\colon\dP\to(\dP\cap V_{\alpha})$ such that whenever $X\subseteq\dP$ is directed, $|X|<\mu$ and $p\in\dP\cap V_{\alpha}$ satisfies $p\leq\pi(q)$ for every $q\in X$, there is $p^*\in\dP$ such that $\pi(p^*)=p$ and $p\leq q$ for every $q\in X$.
    \end{enumerate}
\end{mydef}
\begin{myrem}
 $\Coll(\mu,{<}\nu)$ is a $\LIP(\mu,\nu)$-forcing whenever $\nu$ is measurable.
\end{myrem}

The proof of the following is as in \cite{GalvinJechMagidorIdealGame} and was first noticed by Laver:

\begin{mylem}[{\cite[Lemma 3.3]{JakobTotalFailure}}]\label{lemma: LIP-forcing forces LIP}
    Let $\mu<\nu$ be regular cardinals such that $\nu$ is measurable. Let $\dP$ be a $\LIP(\mu,\nu)$-forcing. Then $\dP$ forces $\LIP(\mu,\nu)$. \qed
\end{mylem}

We will later consider an iteration of $\LIP$-forcings. However, for technical reasons, we will need to show that the corresponding termspace forcings (Definition~\ref{def: Termspace}) also force $\LIP$. This is provided by the following lemma:

\begin{mylem}\label{lemma: Termspace of LIP is LIP}
    Let $\dP$ be a poset and let $\dot{\dQ}$ be a $\dP$-name for a poset. Let $\mu<\nu$ be regular cardinals such that $\nu$ is measurable. Assume that $|\dP|<\nu$ and $\dP$ forces that $\dot{\dQ}$ is a $\LIP(\check{\mu},\check{\nu})$-forcing. Then $\dT(\dP,\dot{\dQ})$ is a $\LIP(\mu,\nu)$-forcing.
\end{mylem}

\begin{proof}
    First of all, $\dT(\dP,\dot{\dQ})$ is ${<}\,\mu$-directed closed by the \emph{Maximum Principle} and separative by Lemma \ref{lemma: Termspace separative}. The small size of $\dP$ easily implies that $\dT(\dP,\dot{\dQ})\subseteq V_{\nu}$. Additionally, in combination with the measurability of $\nu$, it also implies that $\dT(\dP,\dot{\dQ})$ is $\nu$-cc: Assume that $\mathcal{A}\subseteq\dT(\dP,\dot{\dQ})$ is an antichain with size $\nu$. For $\dot{q},\dot{q}'\in\mathcal{A}$, there exists $p\in\dP$ such that $p\Vdash\dot{q}\perp\dot{q}'$. (Otherwise, we would have $1_{\dP}\Vdash\dot{q}||\dot{q}'$ and could apply the maximum principle to find a condition witnessing the compatibility.) So we have a map from $[\mathcal{A}]^2\to\dP$. Since $|\mathcal{A}|=\nu$, which is measurable, and $|\dP|<\nu$, there is $\mathcal{B}\subseteq\mathcal{A}$ with size $\nu$ and $p\in\dP$ such that for every $\dot{q},\dot{q}'\in\mathcal{B}$, $p\Vdash\dot{q}\perp\dot{q}'$. But then $p$ forces that $\mathcal{B}$ is an antichain in $\dot{\dQ}$, a contradiction.

    So we have to find the projection. Again, the small size of $\dP$ implies that for almost all $\alpha$, $\dT(\dP,\dot{\dQ}\cap\dot{V}_{\alpha})=(\dT(\dP,\dot{\dQ}))\cap V_{\alpha}$. So if we let $\dot{\pi}_{\alpha}$ be a $\dP$-name for the projection $\dot{\dQ}\to\dot{\dQ}\cap\dot{V}_{\alpha}$, we can map $\dot{q}$ to $\dot{\pi}_{\alpha}(\dot{q}_{\alpha})$ and obtain a projection from $\dT(\dP,\dot{\dQ})$ to $\dT(\dP,\dot{\dQ})\cap V_{\alpha}$. (Here we again use the \emph{Maximum Principle}.) Let $X\subseteq\dT(\dP,\dot{\dQ})$ be directed, $|X|<\mu$ and $\dot{q}\in\dT(\dP,\dot{\dQ})\cap V_{\alpha}$ with $\dot{q}\prec\dot{\pi}_\alpha(\dot{r})$ for every $\dot{r}\in X$. It follows that $1_{\dP}$ forces that $\check{X}$ is directed and $\dot{q}\dot{\leq}\dot{\pi}_{\alpha}(\dot{r})$ for every $\dot{r}$ in $X$. By the maximum principle we can fix a $\dP$-name $\dot{q}^*$ such that $1_{\dP}$ forces $\dot{\pi}_{\alpha}(\dot{q}^*)=\dot{q}$ and $\dot{q}^*\leq\dot{r}$ for every $\dot{r}\in \check{X}$. Then $\dot{q}^*$ is easily seen to be as required.
\end{proof}

We now define the following enhancement of $\LIP$ which is necessary for the definition of the main Prikry-type forcing of this section:
\begin{mydef}
 An increasing sequence of regular cardinals $\langle\mu\rangle^{\frown}\langle \kappa_i\mid i<\gamma\rangle$  with $\gamma<\mu$ limit 
 witnesses the \emph{Simultaneous Laver Ideal Property}, $\SLIP(\mu,\langle \kappa_i\mid {i<\gamma}\rangle)$, if the following properties hold true for each ordinal  $i<\gamma$:
    \begin{enumerate}
        \item  $\langle \kappa_i\mid i<\gamma\rangle$ is discrete; to wit, $2^{\sup_{j<i}\kappa_{j}}<\kappa_{i}$ holds for all $i<\gamma$.
        \item There is a ${<}\kappa_i$-complete ideal $\mathcal{I}_i$ over $\kappa_i$ and a dense set $\mathcal{B}_i\s \mathcal{I}_i^+$  with respect to $\s$ such that the poset $(\mathcal{B}_i,\s)$ is ${<}\mu$-directed closed.
        \item For every $i<\gamma$, $(\prod_{j\in [i, \gamma)}\mathcal{B}_j,\s)$ is ${<}(2^{\sup_{\ell<i}\kappa_\ell})^+$-distributive.\footnote{Our convention here is that $(2^{\sup_{j<0}\kappa_j})^+:=\mu$.}
    \end{enumerate}
\end{mydef}
\begin{myrem}
    We emphasize that $\gamma$ is not assumed to be regular, but just a limit ordinal. This will be useful later to get instances of failure of the approachability property at successors of singulars of the form $\aleph_{\omega+\omega}$ or $\aleph_{\omega_1+\omega}$.
\end{myrem}

The previous property is a weakening of the requirement that $\LIP(\kappa_i^+,\kappa_{i+1})$  holds for every ordinal $i<\gamma$. The point is that in the construction of the forcing iteration leading to Theorem~\ref{ThmA} we will only be able to establish that the weaker property $\SLIP(\mu, \langle \kappa_i\mid {i<\gamma}\rangle)$ holds. As a result, we have to carry out our proofs using this weaker assumption. The basic reason behind this is that, for an increasing sequence $\langle\kappa_i\mid i<\gamma\rangle$ of measurable cardinals and $\mu<\kappa_0$ regular, the iteration of Levy-collapses collapsing $\kappa_0$ to $\mu$ and everything between $(\sup_{j<i}\kappa_j)^+$ and $\kappa_i$ to $(\sup_{j<i}\kappa_j)^+$ forces $\LIP((\sup_{j<i}\kappa_j)^+,\kappa_i)$; in constrast,  the product of the corresponding  Levy-collapses only forces $\SLIP(\mu,\langle \kappa_i\mid {i<\gamma}\rangle)$:

\begin{mylem}\label{lemma: Product forces LIP}
    Let $\gamma$ be a limit ordinal and let $\langle \kappa_i\mid i<\gamma\rangle$ be an increasing sequence of measurable cardinals. Let $\mu<\kappa_0$. For each $i<\gamma$, let $\dP_i$ be a $\LIP((2^{\sup_{j<i}\kappa_j})^+,\kappa_i)$-forcing. Then the full-support product $\dQ:=\prod_{i<\gamma}\dP_i$ forces $\SLIP(\mu,\langle\kappa_i\mid i<\gamma\rangle)$.
\end{mylem}

\begin{proof}
    Let $G$ be $\dQ$-generic and denote $\vec\kappa=\langle \kappa_i\mid i<\gamma\rangle.$ It is clear that, in $V[G]$, for each $i<\gamma$, $2^{\sup_{j<i}\kappa_j}<\kappa_i$, just by the closure of the posets. For each $i<\gamma$, let $G_i$ be the $\dP_i$-generic filter induced by $G$. In $V[G_i]$, let $\mathcal{J}_i$ and $\mathcal{B}_i$ witness $\LIP(2^{(\sup_{j<i}\kappa_j)^+},\kappa_i)$. (This holds by Lemma~\ref{lemma: LIP-forcing forces LIP}.) In $V[G]$, let $\mathcal{I}_i$ be the ideal generated by $\mathcal{J}_i$. We claim that $(\mathcal{I}_i,\mathcal{B}_i)_{i<\gamma}$ witness $\SLIP(\mu, \vec\kappa)$ in $V[G]$. 
    
    For each $\zeta<\gamma$, define $\dQ_{\zeta}:=\prod_{i<\zeta}\dP_i$, and $\dQ^{\zeta}:=\prod_{i\in[\zeta,\gamma)}\dP_i$, so that $$\dQ\simeq \dQ_{\zeta}\times\dQ^{\zeta}.$$

    \setcounter{myclaim}{0}

    \begin{myclaim}
        Let $i<\gamma$. Then $\mathcal{I}_i$ is ${<}\,\kappa_i$-complete in $V[G]$.
    \end{myclaim}

    \begin{proof}
By definition, $\mathcal{I}_i=\{A\in \mathcal{P}^{V[G]}(\kappa_i)\mid \exists B\in \mathcal{J}_i\, (A\s B)\}$. Let $(A_\alpha)_{\alpha<\theta<\kappa_i}$ be a sequence in $V[G]$ consisting of members of $\mathcal{I}_i$. Without loss of generality, $A_\alpha\in\mathcal{J}_i$ for all $\alpha<\theta$. Since $V[G]$ is an extension of $V[G_i]$ using the product $\dQ_i\times\dQ^{i+1}$ and $\dQ^{i+1}$ is ${<}\,\kappa_i^+$-closed in $V$, and thus ${<}\,\kappa_i^+$-distributive in $V[G_i]$, it follows that $(A_\alpha)_{\alpha<\theta<\kappa_i}\in V[\prod_{j\leq i}G_j].$ On the other hand,  $\mathbb{Q}_i$ is a $\kappa_i$-cc poset in $V[G_i]$ (in fact, it has cardinality $\leq 2^{\sup_{j<i}\kappa_j}$, which is below $\kappa_i$). As a result, we can find a sequence $(X_\alpha)_{\alpha<\theta}\in V[G_i]$ consisting of $X_\alpha\s [\mathcal{J}_i]^{<\kappa_i}$ such that $A_\alpha\in X_\alpha$. By the ${<}\kappa_i$-completeness of $\mathcal{J}_i$ in $V[G_i]$ we deduce that $\bigcup_{\alpha<\theta} A_\alpha\in \mathcal{J}_i\s \mathcal{I}_i$.
    \end{proof}

    The more substantial claim is the following:

    \begin{myclaim}
        Let $\zeta<\gamma$ and $i\in[\zeta,\gamma)$. Let $H^{\zeta}$ be the $\dQ^{\zeta}$-generic filter induced by $G$. In $V[H^{\zeta}]$, $\mathcal{B}_i$ is dense in the dual of the ideal generated by $\mathcal{J}_i$ and ${<}\,(2^{\sup_{j<\zeta}\kappa_j})^+$-directed closed.
    \end{myclaim}

    \begin{proof}
        Let $\mathcal{K}_i$ be the ideal generated by $\mathcal{J}_i$ in $V[H^{\zeta}]$. The model $V[H^{\zeta}]$ is an extension of $V[G_i]$ using $\prod_{j\in[\zeta,\gamma),\zeta\neq i}\dP_i=\prod_{j\in[\zeta,i)}\dP_j\times\prod_{j\in(i,\gamma)}\dP_j$. The second part again adds no new subsets of $\kappa_i$ and thus does not violate the density of $\mathcal{B}_i$. The first part has size ${<}\,\kappa_i$ and thus it also does not violate the density of $\mathcal{B}_i$ as shown by the next argument:  Whenever $\tau$ is a $\prod_{j\in[\zeta,i)}\dP_j$-name for an element of $\dot{\mathcal{K}}^+$, forced by $p$, there is $q\leq p$ such that the set of all elements of $\kappa_i$ forced by $q$ to be in $\tau$ is in $\mathcal{J}_i$, by the completeness of $\mathcal{J}_i$ (i.e., $\kappa_i$-completeness) and the small size of the poset (i.e., ${<}\kappa_i$). Ergo we can find $X\in\mathcal{B}_i$ such that $q\Vdash\check{X}\subseteq\tau$. 

{Since $V[H^{\zeta}]$ is an extension of $V[G_i]$ using a ${<}\,(2^{\sup_{j<\zeta}\kappa_j})^+$-closed poset, $\mathcal{B}_i$ is still ${<}\,(2^{\sup_{j<\zeta}\kappa_j})^+$-directed closed in that model.}
    \end{proof}

    Thereby, every $\mathcal{B}_i$ is ${<}\,\mu$-directed closed in $V[G]=V[H^0]$. For the distributivity of $(\prod_{j\in [\zeta,\gamma)}\mathcal{B}_j, \s)$ we argue as follows. Fix $\zeta<\gamma$. In $V[H^{\zeta}]$, the poset $$\textstyle (\prod_{i\in[\zeta,\gamma)}\mathcal{B}_i,\s)$$ is ${<}\,(2^{\sup_{j<\zeta}\kappa_j})^+$-directed closed. Since $V[G]$ is an extension of $V[H^{\zeta}]$ using $\dQ_{\zeta}$ (a poset  with the ${<}(2^{\sup_{j<\zeta}\kappa_j})^+$-cc),  Easton's lemma (Lemma~\ref{lemma: Eastons lemma}) implies that, in $V[G]$, $\prod_{i\in[\zeta,\gamma)}\mathcal{B}_i$ is still ${<}\,(2^{\sup_{j<\zeta}\kappa_j})^+$-distributive.
\end{proof}

\begin{myrem}
In  Section~\ref{sec: failure everywhere}, we will assume that the GCH holds and our sequence $\langle \kappa_i \mid i < \gamma \rangle$ will be \textbf{continuous} (i.e. only those $\kappa?i$ with $i$ a successor will be measurable). Consequently, there we will be concerned with $\LIP(\kappa_i^{++}, \kappa_{i+1})$-forcing notions.
\end{myrem}

\subsection{The main forcing}

Our setup assumptions for the rest of the section are:
\begin{mysetup}
    We assume that $\langle \mu\rangle^{\smallfrown}\langle \kappa_i\mid i<\gamma\rangle$ is an increasing sequence of regular cardinals witnessing $\SLIP(\mu,\langle \kappa_i\mid i<\gamma\rangle).$ We set $\vec\kappa=\langle\kappa_i\mid i<\gamma\rangle$ and $\kappa$ for the limit of this sequence. Finally, we denote the witnessing sequence of ideals and (subsets) of positive sets as $\vec{\mathcal{I}},\vec{\mathcal{B}}$, respectively.
\end{mysetup}

The main Prikry-type poset of the section is:

\begin{mydef}
  Let $\mathbb{P}(\mu,\vec\kappa,\vec{\mathcal{I}},\vec{\mathcal{B}})$ be the poset consisting of sequences
  $$\textstyle p=\langle p_{i}\mid i<\gamma\rangle\in\prod_{i<\gamma}(\kappa_{i}\cup \mathcal{B}_{i})$$
  whose \emph{support}  $\supp(p):=\{i<\gamma\mid p_{i}\notin \mathcal{B}_{i}\}$ is finite.

  \smallskip

  Given conditions $p, q$, write $p\leq q$ if $\supp(p)\supseteq\supp(q)$ and for each $i<\gamma$:
  \begin{enumerate}
      \item  $p_{i}=q_{i}$ if $i\in\supp(q)$;
      \item $p_{i}\in q_{i}$ if $i\in\supp(p)\smallsetminus\supp(q)$;
      \item $p_{i}\s q_{i}$ otherwise.
  \end{enumerate}
  We write $p\leq_0q$ if $p\leq q$ and $\supp(p)=\supp(q)$.
\end{mydef}

This poset can be viewed as a form of a Laver-style Namba forcing. However, it is defined as a product instead of possessing the form of a tree. This is because of the following problems: On one hand, if we take the version of higher Namba forcing where the trees have height some cardinal $\kappa$ and any stem of length ${<}\,\kappa$, it is unclear whether this forcing has the Prikry property. On the other hand, if we take trees of height some cardinal $\kappa$ which split at cofinitely many levels, already the first branch of length $\omega$ is generic and thus does not have a defined successor set. Thanks to the discreteness of $\vec\kappa$ and the ``closure properties" of $\vec{\mathcal{I}}$ and $\vec{\mathcal{B}}$ we will be able to identify $(\mathbb{P}(\mu,\vec\kappa,\vec{\mathcal{I}},\vec{\mathcal{B}}),\leq,\leq_0)$ as a Prikry-type forcing. However, the arguments will be much more technical.

\smallskip

The poset $\mathbb{P}$ is quite similar in its design to the poset used by Gitik \cite{GitikOverlapping}. The main difference is that here we are using ideals instead of ultrafilters but we can circumvent this hurdle by taking really good care in the arguments. 

\smallskip

The following is an immediate consequence of $\vec{\mathcal{B}}$ being ${<}\mu$-directed closed:

\begin{mylem}
   $\langle \mathbb{P}(\mu,\vec\kappa,\vec{\mathcal{I}},\vec{\mathcal{B}}), \leq_0\rangle$ is ${<}\mu$-directed closed and ${<}\kappa_0$-closed. \qed
\end{mylem}
\begin{myrem}\label{rem: mu directedness}
    The ${<}\mu$-directed-closure of the poset does not play any role in the analysis of this section. It will critical in Section~\ref{sec: failure everywhere} when we define the iteration yielding Theorem~\ref{ThmD}. Indeed, the directed closure is one of the cruxes of Theorem~\ref{Theorem: Indestructibility of neg AP} that ensures that the failure of $\AP$ is preserved during the iteration.
\end{myrem}

From basic cardinal computations we deduce the following:

\begin{mylem}
  Assume the $\mathrm{GCH}$. Then  $\mathbb{P}(\mu,\vec\kappa,\vec{\mathcal{I}},\vec{\mathcal{B}})$ has the $\kappa^{++}$-cc. \qed
\end{mylem}

\smallskip

Now turn to define the minimal extensions of a condition:

\begin{mydef}
    Let $p\in \mathbb{P}(\mu,\vec\kappa,\vec{\mathcal{I}},\vec{\mathcal{B}})$, $i\notin \supp(p)$ and $\alpha\in p(i)$. The one-point extension of $p$ by $\alpha$, $p\cat(i,\alpha)$, is the sequence $q=\langle q_j\mid j<\gamma\rangle$ defined as follows:
    \begin{enumerate}
        \item $\supp(q)=\supp(p)\cup \{i\}$.
        \item $q_j:=\begin{cases}
            p_j, & \text{if $j\in\supp(p)$;}\\
            \alpha, & \text{otherwise.}
        \end{cases}$ 
    \end{enumerate}
    Given a finite set $I\s \gamma\setminus \supp(p)$ and  ordinals $\vec\alpha=\langle \alpha_0,\dots, \alpha_{|I|-1}\rangle\in \prod_{i\in I} p(i)$, the sequence $p\cat (I, \vec\alpha)$ is defined by recursion in the obvious manner.
\end{mydef}
\begin{mylem}
    Let $p\in \mathbb{P}(\mu,\vec\kappa,\vec{\mathcal{I}},\vec{\mathcal{B}})$, $I\s \gamma\setminus \supp(p)$ and ordinals $\vec\alpha\in \prod_{i\in I}p(i)$. Then, $p\cat (I,\vec\alpha)$ is a condition in $\mathbb{P}(\mu,\vec\kappa,\vec{\mathcal{I}},\vec{\mathcal{B}})$ and $p\cat(I,\vec\alpha)\leq p$ \qed
\end{mylem}
The next easy observation will be critical later on: 
\begin{mylem}\label{lemma: not too many minimal extensions}
    Let $p\in \mathbb{P}(\mu,\vec\kappa,\vec{\mathcal{I}},\vec{\mathcal{B}})$ and $I\s \gamma\setminus \supp(p)$ finite. Then,
    $$\textstyle\{p\cat (I,\vec\alpha)\mid \vec\alpha\in \prod_{i\in I}p(i)\}$$
   is a maximal antichain $\leq$-below $p$ with size ${<}\kappa_{\max(I)+1}.$ \qed
\end{mylem}

\begin{myconv}
  To simplify notations hereafter we  write $\mathbb{P}$ instead of $\mathbb{P}(\mu,\vec\kappa,\vec{\mathcal{I}},\vec{\mathcal{B}})$. 
\end{myconv}

In our work, the poset $\dP$ plays a similar role to that of Cohen forcing in the original Mitchell poset (see \cite{MitchellTreeProp}): On its own, the Levy-collapse $\Coll(\mu,\kappa^+)$ has the property that every initial segment of the added surjection from $\mu$ onto $\kappa^+$ lies in the ground model (by the ${<}\,\mu$-closure of the poset). However, Mitchell noticed that by using Cohen forcing $\Add(\tau,1)$ (for $\tau<\mu$) \emph{before} the Levy-collapse, the iteration $\Add(\tau,1)*\dot{\Coll}(\check{\mu},\check{\kappa}^+)$ has the opposite property: Every new function from $\mu$ into the ordinals has an initial segment which does not lie in the ground model. This is key to his proof that in a Mitchell-generic extension -- collapsing a Mahlo cardinal to become $\mu^+$ -- $\AP_{\mu}$ fails. In our case, we will be obtaining an even stronger statement (for a smaller class of sequences): After forcing with the iteration $\dP*\dot{\Coll}(\check{\mu},\check{\kappa}^+)$, every surjection from $\mu$ into $\kappa^+$ has an initial segment \emph{which is not even covered by a ground-model set of size ${<}\,\kappa$}. This will be shown in Lemma \ref{label: key for non approachability} as a culmination of a \emph{Prikry analysis} of $\dP$ and is key for our goal which is showing that $\dP*\dot{\Coll}(\check{\mu},\check{\kappa}^+)$ does not make $\kappa^+$ $d$-approachable with respect to any normal coloring $d$ on $\kappa^+$ in the ground model.

\smallskip

\smallskip

Let us now prove that this poset works as intended. We first have to show that the poset has the \emph{Strong Prikry Property}. From this we will derive both the Prikry Property and  the $\kappa_0$-pure decidability property (see Definition~\ref{def: pure decidability}). 

One notable difference between this poset and the Laver-style Namba forcing considered in \cite[\S4]{JakobLevine} is that the latter can eventually decide names for ordinals below any $\kappa_n$. (Because any condition will eventually have a sufficiently long stem.) This implies that the poset does not add bounded subsets of $\kappa$. However, since we are using conditions of potentially uncountable length we will also be e.g. adding new subsets of $\sup_n\kappa_n$ and thus bounded subsets of $\kappa$.

\begin{mylem}\label{lemma: SPP}
    $(\mathbb{P}, \leq, \leq_0)$ has the Strong Prikry property. Namely, for each condition $p\in \mathbb{P}$  and a dense open set $D\s \mathbb{P}$   there is $q\leq_0 p$ and $I\subseteq \gamma\setminus \supp(q)$ finite such that for each $\vec\alpha\in \prod_{i\in I} q(i)$, $q\cat (I,\vec\alpha)\in D.$
\end{mylem}

\begin{proof}
    Given a condition $r\in \mathbb{P}$ we will say that $r$ is \emph{good} (resp. \emph{really good}) if it witnesses the conclusion of the lemma (resp. by taking $q=r$). If $r$ is not good we will say that  $r$ is \emph{bad}. Note that if $r$ is bad any $\leq_0$-extension of it is bad as well.

    \smallskip

    Toward a contradiction, let us assume that the given condition $p$ is bad.
    
    \setcounter{myclaim}{0}
    
    \begin{myclaim}\label{claim: spp}
        Given an index $i\notin\supp(p)$ there is $q\leq_0 p$ such that $q\restriction i=p\restriction i$ and $q\cat (i, \alpha)$ is bad for all $\alpha\in q(i).$
    \end{myclaim}
    
    \begin{proof}[Proof of claim]
       Let $W$ denote the positive set $p(i)$. For each $\alpha\in W$ the set $D_\alpha$ consisting of all $q\in \prod_{j\in (i,\gamma)} \mathcal{B}_j$ such that either $(p\restriction i)^\smallfrown \langle \alpha\rangle^\smallfrown q$ is bad or there is $r\leq_0 p\restriction i$ such that $r^\smallfrown \langle \alpha\rangle^\smallfrown q$ is really good is dense open. (Note that this may be false if we replace ``really good'' by ``good''.)
       By ${<}\kappa_i^+$-distributivity of the poset  $(\prod_{j\in (i,\gamma)} \mathcal{B}_j,\s)$, the set $D=\bigcap_{\alpha\in W} D_\alpha$ is dense in $\prod_{j\in (i,\gamma)} \mathcal{B}_j$ as well. Let $q^*\in D.$

       \smallskip
       
        We split $W$ into two disjoint sets as follows:
        $$W_{0}:=\{\alpha\in W\mid (p\restriction i)^\smallfrown\langle \alpha\rangle^\smallfrown q^*\,\text{is bad}\},$$
        $$W_{1}:=\{\alpha\in W\mid \exists r\leq_0(p\restriction i)\, (r^\smallfrown\langle \alpha\rangle^\smallfrown q^*\,\text{is really good})\}.$$
        Suppose that $W_0\in \mathcal{I}_i^+$.  Then, there is $Z\in \mathcal{B}_i$ such that $Z\s W_0$ (by density of $\mathcal{B}_i$). Let $q$ be defined as $(p\uhr i)^{\frown}Z^{\frown}q^*\leq_0 p$. Clearly, $q$ is as wished.

        \smallskip

        Suppose otherwise that $W_1\in \mathcal{I}_i^+$. For each $\alpha\in W_1$ there is a condition $r_\alpha$ and a finite set $I_\alpha\s\gamma$ witnessing the property. Note that there are at most $2^{\sup_{j<i}\kappa_j}$-many of such $r_\alpha$'s and by assumption $2^{\sup_{j<i}\kappa_j}<\kappa_i$. Additionally, there are at most $\gamma$ many such $I_{\alpha}$'s and $\gamma<\kappa_0$. Hence, by ${<}\kappa_i$-completeness of $\mathcal{I}^+_i$ and density of $\mathcal{B}_i$, there is $Z\s W_1$ in $\mathcal{B}_i$ and $r^*$, $I^*$ such that $r^*=r_\alpha$ and $I_{\alpha}=I$ for all $\alpha\in Z$. Let $q={r^*}^\smallfrown \langle Z\rangle^\smallfrown q^*$. Clearly, $q\leq_0 p$. A moment of reflection makes clear that $p$ is good as witnessed by $q$ and $I=I^* \cup\{i\}$. But we were assuming that $p$ was bad, so $W_1$ cannot be $\mathcal{I}_i$-positive, which rules out this second alternative.

        \smallskip

        Therefore, it has to be the case that $W_0\in\mathcal{I}_i^+$ and so $q$ is as wished.
    \end{proof}
    
    Without loss of generality, suppose that $\supp(p)=\emptyset$. Now we use the previous claim to survey over all possible extensions of $p$. More precisely, we define, by induction, a $\leq_0$-decreasing sequence  $\langle q_\delta\mid \delta<\gamma\rangle$ in $\mathbb{P}$ with the property that for each $I\in [\delta]^{<\omega}$ and $\vec\alpha\in \prod_{i\in I}q_\delta(i)$, $q_\delta\cat (I,\vec{\alpha})$ is bad.

    \smallskip

    Suppose we have succeeded in defining $\langle q_\delta\mid \delta<\epsilon\rangle$. 
    
    \smallskip
    
    \underline{Case $\epsilon$ is limit:} In this case we simply take a $\leq_0$-lower bound $q_{\epsilon}$ for this sequence. Notice that this choice works: Given any $I\in [\epsilon]^{<\omega}$ and $\vec\alpha\in \prod_{i\in I}q_{\epsilon}(i)$ there is an index $\delta<\epsilon$ such that $q_{\epsilon}\cat (I,\vec\alpha)\leq_0 q_\delta\cat (I,\vec{\alpha})$ and the latter condition is bad. By transitivity of the property of being bad, it follows that  $q_{\epsilon}\cat (I,\vec\alpha)$ is bad as well.

    \smallskip

    \underline{Case $\epsilon$ is successor:} Suppose that $\epsilon=\bar\epsilon +1$. Let $\{(I_i, \vec\alpha_i)\mid i<\theta\}$ be an enumeration of all possible finite subsets of $\bar\epsilon$ and $\vec\alpha_i\in \prod_{j\in I} q_{\bar\epsilon}(j)$. 
    
    For each $i<\theta$ let us consider the set of all possible ``bad tails":
    $$\textstyle D_i=\{q\in \prod_{j\in [\bar{\epsilon},\gamma)}\mathcal{B}_j\mid \forall \beta\in q(\bar\epsilon)\, ((q_{\bar\epsilon}\restriction\bar\epsilon ^\smallfrown q)\cat(I_i,\vec\alpha_i)\cat (\bar\epsilon,\beta)\,\text{ is bad})\}.$$
    This set is clearly dense in $(\prod_{j\in [\bar{\epsilon},\gamma)}\mathcal{B}_j,\s)$ as per the previous claim.

    Note that $\theta\leq\sup_{\eta<\bar{\epsilon}}\kappa_{\eta}$, since there are at most $|\bar{\epsilon}|$-many finite subsets of $\bar\epsilon$ and any $\vec\alpha_i$ is a function from $I_i$ into $\sup_{\eta<\bar{\epsilon}}\kappa_{\eta}$. As a result, we can let $q\in \bigcap_{i<\theta} D_i$, which exists by ${<}(\sup_{\eta<\bar\epsilon}\kappa_{\eta})^+$-distributivity. 
    Clearly, $p_\epsilon=(p_{\bar\epsilon}\restriction\bar\epsilon)^\smallfrown  q$ is as desired.

   \smallskip

   After this construction we obtain a $\leq_0$-decreasing sequence  $\langle p_\delta\mid \delta<\gamma\rangle$. In the end, we let $p^*$ be a $\leq_0$-lower bound for this sequence (which exists because $\gamma<\kappa_0$). As in the limit case before,  $p^*\cat (I, \vec{\alpha})$ is bad for every  $I\in [\gamma]^{<\omega}$ and $\vec\alpha\in \prod_{i\in I}p^*(i)$.

   \smallskip

  By density of $D$ there is $q\leq p^*$ inside it. Let $(I,\vec\alpha)$ be such that $q\leq_0 p^*\cat (I,\vec\alpha).$ By the previous comments, $p^*\cat (I,\vec\alpha)$ is bad, and as a result so is $q$ as well. However, $q\in D$, so it must be really good. This yields the desired contradiction.
\end{proof}
The previous proof in fact gives the following \emph{fusion-like} Strong Prikry Property:
\begin{mylem}
 For each condition $p\in \mathbb{P}$,  $i\in (\max(\supp(p)), \gamma)$ and a dense open set $D\s \mathbb{P}$   there is $q\leq_0 p$ with $q\restriction i= p \restriction i$ and $I\subseteq \gamma\setminus \supp(q)$ finite such that for each $\vec\alpha\in \prod_{i\in I} q(i)$, $q\cat (I,\vec\alpha)\in D.$  
\end{mylem}
\begin{proof}[Proof sketch]
    Run the same argument as in the previous lemma with the only difference that in the construction  following Claim~\ref{claim: spp} the conditions $q_\delta$ are taken in such a way that they agree up to coordinate $i$.
\end{proof}
For each $i<\gamma$ let $\leq^i_0$ denote the subordering of $\leq_0$ defined by $$\text{$p\leq^i_0 q\:\Leftrightarrow p\leq_0 q$ and $p\restriction i=q\restriction i.$}$$
Note that $(\mathbb{P},\leq^i_0)$ is ${<}\kappa_i$-closed. This has the following important consequence: 
\begin{mylem}\label{lemm: fusion SPP}
    Forcing with $\mathbb{P}$ preserves $\kappa^+$.
\end{mylem}
\begin{proof}
    Suppose that this is not the case. Then, there is a condition $p\in \mathbb{P}$, a regular cardinal $\lambda<\kappa$ and a $\mathbb{P}$-name $\dot{f}$ such that $p\forces``\dot{f}\colon \check{\lambda}\rightarrow \check{\kappa}^+$ is cofinal". Let $i<\gamma$ be an index such that $\kappa_i>\lambda$. For each $\alpha<\lambda$ let $D_\alpha$ be the dense set of conditions in $\mathbb{P}$ deciding a value for $\dot{f}(\alpha)$. Using Lemma~\ref{lemm: fusion SPP} and the closure of $\leq^i_0$ we can find a $\leq^i_0$-decreasing sequence of conditions $\langle p_\alpha\mid\alpha\leq \lambda\rangle$ below $p$ which for each $\alpha<\lambda$ there is a finite set $I_\alpha\s \gamma\setminus \supp(p)$ such that $p_\alpha\cat (I_\alpha,\vec\theta)\in D_\alpha$ for all $\vec\theta\in \prod_{i\in I_\alpha}p_\alpha(i)$. For each $\alpha<\lambda$ let $\Delta_\alpha$ be the set of all possible decisions; namely,
    $$\textstyle\Delta_\alpha=\{\beta<\kappa^+\mid \exists \vec\theta\in \prod_{i\in I_\alpha}p_\alpha(i)\, (p_\alpha\cat (I_\alpha, \vec\theta)\forces \dot{f}(\check{\alpha})=\check{\beta})\}.$$
This set is of cardinality ${<}\kappa^+$. Therefore, there is $\rho_\alpha<\kappa^+$  such that $\sup\Delta_\alpha<\rho_\alpha$. By regularity of $\kappa^+$ (in $V$) the supremum of all of these $\rho_\alpha$'s (call it $\rho$) must be below $\kappa^+$. We claim that $p_\lambda$ forces $``\mathrm{range}(\dot{f})\s \check{\rho}$" which will yield a contradiction.

Let $\alpha<\lambda$ and $q\leq p_\lambda$ deciding the value of $\dot{f}(\check{\alpha})$. Without loss of generality, $q\leq p_\alpha\cat (I_\alpha,\vec\theta)$ for some of such $\vec\theta$. Therefore, $q\forces\dot{f}(\check{\alpha})<\check{\theta}_\alpha<\check{\theta}.$ By density, $p_\lambda$ has the sought property, which completes the proof.
\end{proof}

We can derive the Prikry property from the strong Prikry property. We prove that $(\mathbb{P},\leq, \leq_0)$ has pure ${<}\kappa_0$ decidability which clearly implies that $(\mathbb{P},\leq,\leq_0)$ has the Prikry property.
\begin{mylem}\label{lemma: P has Prikry prop}
    Let $p\in \mathbb{P}$ be a condition  and $\tau$ a $\mathbb{P}$-name such that $p\forces_{\mathbb{P}}\tau<\check{\mu}$ for some $\mu<\kappa_0$. Then, there is $p^*\leq_0 p$ and $\delta<\kappa_0$ such that $p^*\forces_{\mathbb{P}}\tau=\check{\delta}.$
\end{mylem}
\begin{proof}
    Let $D=\{q\in \mathbb{P}\mid q\perp p\, \vee\, \exists\delta\, (q\forces_{\mathbb{P}}\tau=\check{\delta})\}.$ This is clearly dense open, so that there is $p^*\leq_0 p$ and $I\s \gamma$ finite such that $p^*\cat (I,\vec\alpha)\in D$ for all $\vec\alpha\in \prod_{i\in I}p^*(i)$. Suppose that $q$ was chosen so  that $I$ is as small as possible.

    We claim that $I=\emptyset$ thus establishing the lemma. Suppose otherwise, and set $J:=I\setminus \{\max(I)\}$. For each ordinal $\beta\in p^*(\max(I))$ let $f_\beta\colon \prod_{j\in J} p^*(j)\rightarrow \mu$ be the function defined as $\vec\alpha\mapsto \delta(\vec\alpha^\smallfrown \langle\beta\rangle)$ where $\delta(\vec\alpha^\smallfrown \langle\beta\rangle)$ is such that $$p^*\cat (I,\vec\alpha^\smallfrown\langle\beta\rangle)\forces \tau=\check{\delta}(\vec\alpha^\smallfrown \langle\beta\rangle).$$

    Since $\vec{\mathcal{I}}(\max(I))$ is a $\kappa_{\max(I)}$-complete ideal and there are ${<}\kappa_{\max(I)}$-many of such functions there is a $\mathcal{I}(\max(I))$-positive set $W\s p^*(\max(I))$ and a function $f$ such that $f_\beta=f$ for all $\beta\in W$. By density of $\mathcal{B}_{\max(I)}$ we can assume that $W\in \mathcal{B}_{\max(I)}$. Let $q\leq_0 p^*$ be the condition obtained after replacing $p^*(\max(I))$ by $W$. It follows that for every $\vec\beta^\smallfrown \langle\alpha\rangle\in \prod_{i\in I} q(i)$, $q\cat (I,\vec\beta^\smallfrown \langle\alpha\rangle)\forces \tau=f(\vec\beta)$. Therefore,  already $q\cat (J,\vec\beta)$ decides $\tau$, which contradicts the minimality of $I$.
\end{proof}

  Let $\dot{\mathbb{Q}}$ be a $\mathbb{P}$-name for a forcing poset. One can naturally define a ``Prikry order" $\leq_0$ on the two-step iteration $ \mathbb{P}\ast \dot{\mathbb{Q}}$ by stipulating that $$\text{$(q,\dot{d})\leq_0 (p,\dot{c})$ if and only if $q\leq_0 p$ and $q\forces_{\mathbb{P}}\dot{d}\leq_{\dot{\mathbb{Q}}} \dot{e}.$}$$
 
\begin{mylem}\label{lemma: SPP for iterations}
 The poset  $(\mathbb{P}\ast \dot{\mathbb{Q}},\leq,\leq_0)$ has the Strong Prikry Property; namely, given $(p,\dot{q})\in \mathbb{P}\ast \dot{\mathbb{Q}}$  and a dense open set $D\s \mathbb{P}\ast \dot{\mathbb{Q}}$   there is $(q,\dot{d})\leq_0 (p,\dot{c})$ and $I\subseteq \gamma\setminus \supp(q)$ finite such that for each $\vec\alpha\in \prod_{i\in I} q(i)$, $(q\cat (I,\vec\alpha),\dot{d})\in D.$
\end{mylem}
\begin{proof}
   Look at $E:=\{q\in \mathbb{P}\mid \exists \dot{d}\, (q,\dot{d})\leq (p,\dot{c})\,\wedge\, (q,\dot{d})\in D\}.$ Clearly, $E$ is dense open for $\mathbb{P}$. By the Strong Prikry Property applied to this poset  we find $q\leq_0 p$ and $I\s \gamma\setminus \supp(q)$ finite such that for each $\vec{\alpha}\in \prod_{i\in I} q(i)$, $q\cat (I,\vec{\alpha})\in E$. This means that for each $\vec\alpha$ there is a $\mathbb{P}$-name $\dot{d}_{\vec\alpha}$ such that $(q\cat (I,\vec{\alpha}),\dot{d}_{\vec\alpha})\in D$. Since $\{q\cat (I,\vec{\alpha})\mid \vec\alpha\in \prod_{i\in I}q(i)\}$ forms a maximal antichain below $q$ we can use the \emph{mixing lemma}  so as to produce a $\mathbb{P}$-name $\dot{d}$ such that $q\cat (I,\vec{\alpha})\forces \dot{d}=\dot{d}_{\vec\alpha}$. A moment's reflection makes clear that $(q,\dot{d})$ is the sought condition.
\end{proof}

Now we can show the key lemma underpinning our analysis of approachability:

\begin{mylem}\label{label: key for non approachability}
    Let $\dot{\dC}$ be a $\mathbb{P}$-name for a $\gamma^+$-closed poset. Let $\dot{F}$ be a $\mathbb{P}*\dot{\dC}$-name for a cofinal function from $\check{\theta}$ into $\check{\kappa}^+$, where $\theta\in (\gamma, \kappa_0)$ regular, forced by a condition $(p,\dot{c})\in\mathbb{P}*\dot{\dC}$. Then there is $(q,\dot{d})\leq_0(p,\dot{c})$ and $\zeta<\theta$ such that for no $(r,\dot{e})\leq (q,\dot{d})$ and $x\in V$ with $|x|<\kappa$, $(r,\dot{e})\Vdash\im(\dot{F}\uhr\check{\zeta})\subseteq\check{x}$.
\end{mylem}

\begin{proof}
    The following is the key claim. Roughly speaking it shows that for any condition $(p',\dot{c}')$ and $i\notin\supp(p')$ there is a  $\leq_0$-extension $(p'',\dot{c}'')$ and $\zeta<\theta$ such that the ``canonical" maximal antichain at coordinate $i$, $\{(p''\cat (i, \beta), \dot{c}'')\mid i\in p''(i)\}$, decides at least $\kappa_i$-many different possibilities for $\dot{F}(\zeta)$. Establishing the existence of such a condition which has that property for every $i$ will show that $\im(\dot{F})$ cannot be covered by a ground model set of size ${<}\kappa$. To simplify notations later on, let us assume that the initial condition $(p,\dot{c})$ is just the trivial condition in $\mathbb{P}\ast\dot{\mathbb{C}}.$
    
    \begin{myclaim}\label{claim: bounding}
        Let $(p',\dot{c}')\in \mathbb{P}\ast \dot{\mathbb{C}}$ and $i\notin\supp(p')$. Then, there is $(p'',\dot{c}'')\leq_0(p',\dot{c}')$,  $\zeta<\theta$ and a sequence $\langle A_\beta\mid \beta\in p''(i)\rangle$ of disjoint sets such that for each $\beta\in p''(i)$, $$(p''\cat (i, \beta), \dot{c}'')\forces \dot{F}(\zeta)\in \check{A}_\beta.$$
    \end{myclaim}
    
    \begin{proof}[Proof of claim]
        The claim is established after combining  the Strong Prikry Property of $\mathbb{P}\ast\dot{\mathbb{Q}}$ (Lemma~\ref{lemma: SPP for iterations}) with the ${<}\,\kappa_i^+$-distributivity of $\mathbb{B}=(\prod_{j\in (i,\gamma)}\mathcal{B}_j,\supseteq).$

        \smallskip

        Towards a contradiction, suppose that the claim was false. We will exhibit a winning strategy for $\mathrm{INC}$ in the completeness game $\mathcal{G}(\mathbb{B},\kappa_i+1)$ (see Definition~\ref{def: INC COM game}). By Foreman's theorem (see Theorem~\ref{Theorem: Foreman Distributivity}), this will yield a contradiction with $\mathbb{B}$ being ${<}\,\kappa_i^+$-distributive.

        \smallskip

        Suppose we have constructed $\langle (A_\beta,\zeta_\beta, r_\beta, \langle p_\beta, q_\beta\rangle, \dot{c}_\beta)\mid \beta<\beta^*\rangle$ such that:
        \begin{itemize}
            \item $A_\beta\s \kappa^+$ of size ${<}\kappa$ and $A_{\beta}\cap A_{\beta'}=\emptyset$ for $\beta\neq\beta'$.
            \item $\zeta_\beta$ is an ordinal below $\theta$.
            \item $r_\beta\leq_0 p'\restriction i$. (I.e., $r_\beta$ is the ``low part" of a condition in $\mathbb{P}$.)
            \item $\langle p_\beta, q_\beta\rangle$ is a move  in the game $\mathcal{G}(\mathbb{B},\kappa_i+1)$.
            \item $(r_\beta^\smallfrown \langle \beta\rangle^\smallfrown q_\beta, \dot{c}_\beta)\forces \dot{F}(\check{\zeta}_\beta)\in \check{A}_\beta.$
        \end{itemize}
        Assume that $\langle (p_\beta, q_\beta)\mid \beta<\beta^*\rangle$ has been already played in $\mathcal{G}(\mathbb{B},\kappa_i+1)$. First $\mathrm{COM}$ moves\footnote{If $\mathrm{COM}$ could not move then we would have already described a winning strategy for $\mathrm{INC}$.} $p_{\beta^*}$ and let $\mathrm{INC}$ makes its move. The second player $\mathrm{INC}$ moves as follows. First, since $((p'\restriction i)^\smallfrown \langle \beta^*\rangle^\smallfrown p_{\beta^*}, \dot{c}')$ forces $``\dot{F}\colon \check{\theta}\rightarrow \check{\kappa}^+$ is cofinal", this condition itself forces $``\exists \zeta<\check{\theta}\, (\dot{F}(\zeta)\notin \bigcup_{\beta<\beta^*}\check{A}_\beta)"$ (note that this  union has cardinality ${<}\kappa^+$).

      Let $(p'_{\beta^*},\dot{c}'_{\beta^*})\leq_0 ((p'\restriction i)^\smallfrown \langle \beta^*\rangle^\smallfrown p_{\beta^*},\dot{c}')$ decide this ordinal $\zeta<\theta$ – denote it $\zeta_{\beta^*}$. It follows that $(p'_{\beta^*},\dot{c}'_{\beta^*})\Vdash\dot{F}(\check{\zeta}_{\beta^*})\notin\bigcup_{\beta<\beta^*}\check{A}_{\beta}$. By  Lemma~\ref{lemma: SPP for iterations} there is a further pure extension  $(p''_{\beta^*},\dot{c}_{\beta^*})\leq_0(p'_{\beta^*},\dot{c}'_{\beta^*})$ and a finite set $I\s \gamma\setminus \supp(p''_{\beta^*})$ such that for every $\vec\alpha\in \prod_{i\in I}p''_{\beta^*}(i)$, $(p''_{\beta^*}\cat (I,\vec{\alpha}),\dot{c}_{\beta^*})$ decides the value of $\dot{F}(\check{\zeta}_{\beta^*})$. Note that the set of decisions
        $$A_{\beta^*}:=\textstyle \{\eta<\kappa^+\mid \exists \vec\alpha\in \prod_{i\in I}p''_{\beta^*}(i)\, (p''_{\beta^*}\cat (I,\vec{\alpha}),\dot{c}_{\beta^*})\forces\dot{F}(\check{\zeta}_{\beta^*})=\check{\eta}\}$$
        is of size ${<}\kappa$ because so is the antichain $\{p''_{\beta^*}\cat (I,\vec{\alpha})\mid  \vec\alpha\in \prod_{i\in I}p''_{\beta^*}(i)\}.$ Also, $$\textstyle A_{\beta^*}\cap\bigcup_{\beta<\beta^*}A_{\beta}=\emptyset.$$

        After this construction, we set $r_{\beta^*}:=p''_{\beta^*}\restriction i$ and let  $\mathrm{INC}$ play $q_{\beta^*}=p''_{\beta^*}\setminus i+1$.

        \medskip

        We claim that the above produces a winning strategy for $\mathrm{INC}$ in $\mathcal{G}(\mathbb{B},\kappa_i+1)$. Suppose that this was not the case and that $\{(p_\beta,q_\beta)\mid \beta<\kappa_i\}$ was a winning play for $\mathrm{COM}$ in the game $\mathcal{G}(\mathbb{B},\kappa_i+1)$. Let $p^*$ be a lower bound for this sequence.  Since $\theta<\kappa_0$ and $2^{\sup_{j<i}\kappa_j}<\kappa_i $ there is $W\s p^*(i)$ in $\mathcal{B}_i$, $\zeta^*<\theta$ and $r^*$ a low part such that  $\zeta_\beta=\zeta^*$ and $r_\beta=r^*$ for all $\beta\in W$. Let $p'':={r^*}^\smallfrown\langle W\rangle ^\smallfrown p^*$. Let $\dot{c}''$ be the mixing name of the $\mathbb{P}$-names $\{\dot{c}_\beta\mid \beta\in W\}$. By construction, it follows that $(p'',\dot{c}'')\leq_0 (p',\dot{c}')$ witnesses together with $(\zeta,\langle A_\beta\mid \beta\in W\rangle)$ the conclusion of the claim. Since we were assuming that no such $\leq_0$-extension of $(p',\dot{c}')$ exists it must be the case that $\mathrm{COM}$ cannot form such a lower bound, $p^*$. In turn, this shows that the above procedure describes a winning strategy for $\mathrm{INC}$ in $\mathcal{G}(\mathbb{B},\kappa_i+1)$. However, this is impossible by Foreman's theorem, which yields the  final contradiction.
    \end{proof}
    
    Next we use the previous claim to complete the proof of the lemma. By induction define a $\leq_0$-sequence of conditions $\langle (p_i, \dot{c}_i)\mid i<\gamma\rangle$ witnessing the previous lemma relative to sets $\langle A^i_\beta\mid \beta\in p_i(i)\rangle$. At limit stages we take a $\leq_0$-lower bound to the previously constructed sequences (which exists in that $\dot{\mathbb{C}}$ is forced to be $\gamma^+$-closed) and then apply Claim~\ref{claim: bounding}.\footnote{Note that in this process we can make sure that the witnessing ordinals $\zeta_i<\theta$ increase.} Let $(q,\dot{d})$ be a $\leq_0$-lower bound for this sequence and set $\zeta:=\sup_{i<\gamma}\zeta_i$. Since $\gamma<\theta$ and $\theta$ is regular it follows that $\zeta<\theta$.

    Suppose towards a contradiction that there is $x\in V$, $|x|<\kappa$ and a condition $(r,\dot{e})\leq (q,\dot{d})$ forcing $``\im(\dot{F}\restriction \check{\zeta})\s \check{x}"$. Let $i<\gamma$, $i\notin \supp(r)$, such that $|x|<\kappa_{i}$. Then, since $r(i)$ is a $\kappa_i$-sized subset of $p_i(i)$ there must be some $\beta\in r(i)$ such that $x\cap A^i_\beta=\emptyset$. However, by our construction,
    $$(r\cat (i,\beta), \dot{e})\leq (p_i\cat (i,\beta),\dot{c}_i)\forces \dot{F}(\check{\zeta}_i)\in\check{A}^i_\beta,$$
    and by assumption, $(r\cat(i,\beta),\dot{e})\vdash\dot{F}(\check{\zeta}_i)\in\check{x}$, a contradiction.
\end{proof}

\begin{mycol}\label{Corollary: Top cardinal not approachable}
    Let $\theta\in (\gamma,\kappa_0)$ be any regular cardinal.  After forcing with $\mathbb{P}*\dot{\Coll}(\check{\theta},\check{\kappa}^+)$, $\kappa^+$ is not $d$-approachable with respect to any normal, subadditive coloring $d\colon[\kappa^+]^2\to\cf(\kappa)$ from the ground model.
\end{mycol}
\begin{proof}
    Suppose that $\kappa^+$ was forced by $\mathbb{P}*\dot{\Coll}(\check{\theta},\check{\kappa}^+)$ to be $d$-approachable. Then, there is a condition $(p,\dot{c})$, a name $\dot{A}$ and an ordinal $j<\cf(\kappa)$ such that $(p,\dot{c})$ forces $``\dot{A}$ is unbounded in $\kappa^+$" and $``\forall \beta_0,\beta_1\in \dot{A}\, (\beta_0<\beta_1\,\to\, \check{d}(\beta_0,\beta_1)\leq \check{j})"$ (see Fact~\ref{fact: characterizing dapproachability}.) Let $\dot{F}\colon \check{\theta}\rightarrow\check{\kappa}^+$ be a $\mathbb{P}*\dot{\Coll}(\check{\theta},\check{\kappa}^+)$-name for the increasing enumeration of $\dot{A}$. By Lemma~\ref{label: key for non approachability} there is $(q,\dot{d})\leq_0 (p,\dot{c})$ and $i<\theta$ such that for no $(r,\dot{e})\leq (q,\dot{d})$ and $x\in V$ with $|x|<\kappa$, $(r,\dot{e})\Vdash\im(\dot{F}\uhr\check{i})\subseteq\check{x}$. By extending $(q,\dot{d})$ is necessary we may in addition assume that it decides the value of $\dot{F}(\check{i})$ – say this is $\xi$. We then deduce that $(q,\dot{d})$ forces $``\im(\dot{F}\restriction\check{i})\s \{\beta<\check{\kappa}^+\mid\check{d}(\beta,\check{\xi})\leq \check{j}\}"$\footnote{Here we use that $d$ is subadditive; namely, that $d(\beta_0,\xi)\leq d(\beta_0,\beta_1)$ for all $\beta_0<\beta_1<\xi.$} but the latter set is, by normality of $d\in V$, a set in $V$ of size $<\kappa.$ A contradiction.
\end{proof}
\begin{myrem}
    Recall that at the beginning of this section we fixed a regular cardinal $\mu\in (\gamma,\kappa_0)$. We mentioned in Remark~\ref{rem: mu directedness} that the importance of this parameter will not be appreciate until we arrive at Section~\ref{sec: failure everywhere} – this is why we omitted it during the discussion we had so far. Nonetheless, note that both Lemma~\ref{label: key for non approachability} and Corollary~\ref{Corollary: Top cardinal not approachable} apply for the particular choice $\theta:=\mu.$
\end{myrem}

\subsection{On Shelah's problem}
In this section we apply the analysis of the previous section to answer Shelah's Question~\ref{que: Shelah}. The proof is inspired by Shelah's \cite{ShelahSuccSingCard}.
\begin{mysen}\label{theorem: One cofinality}
    Assume the $\mathrm{GCH}$ holds. Let $\lambda$ be a supercompact cardinal, $\gamma<\lambda$  a limit ordinal and $\langle \delta_i\mid i<\gamma\rangle$ an increasing  sequence of measurable cardinals above $\lambda$ with limit $\delta$. Let $\mu\in(\cf(\gamma),\lambda)$ be regular. Then, there is a generic extension where $\delta=\lambda^{+\gamma}$ and  $\lambda^{+\gamma+1}\cap \cof(\mu)\notin I[\lambda^{+\gamma+1}]$. 
\end{mysen}

\begin{proof}
    Let $\lambda$, $\mu$ and $\langle \delta_i\mid i<\gamma\rangle$ be as above.  By preparing the universe  we may assume that the supercompactness of $\lambda$ is indestructible under $\lambda$-directed-closed forcings. To avoid potential conflicts with the application of Fact~\ref{fact: eisworth handbook} we do make the supercompactness of $\lambda$ indestructible only under $\lambda$-directed-closed forcings that preserve the GCH. This forcing itself preserves the GCH pattern from the ground model \cite[\S8.1]{PartIII}. 
    Next force with the full support product of Levy-collapses  $$\textstyle \Coll(\lambda,{<}\delta_0)\times\prod_{1\leq i<\gamma}\Coll((\sup_{j<i}\delta_j)^{++},{<}\delta_{i}).$$
    Each of these posets is $\LIP((\sup_{j<i}\delta_j)^{++},{<}\delta_{i})$ by virtue of the measurability of the $\delta_i$'s. Hence,  by Lemma~\ref{lemma: Product forces LIP}, $\SLIP(\lambda,\langle \delta_i\mid i<\gamma\rangle)$ holds in the resulting generic extension. In particular,  $\SLIP(\mu,\langle \delta_i\mid i<\gamma\rangle)$ holds as well. Since $\lambda$ was made indestructible it remains supercompact. Also, $\delta$ becomes $\lambda^{+\gamma}$.

    \smallskip

    Assume that the above model is our ground model, $V$. Now, we use the supercompactness of $\lambda$ to ``reflect" the failure of approachability at $\delta^+$. More precisely:

    \setcounter{myclaim}{0}
    
    \begin{myclaim}
    Let $d\colon [\delta^+]^2\rightarrow \cf(\gamma)$ be a subadditive normal coloring in $V$.
        Then, there is a singular $\kappa\in(\mu,\lambda)$ and a stationary set $S\subseteq E_{\kappa^+}^{\delta^+}$ such that the following hold:
        \begin{enumerate}
            \item[$(1)$] There is an increasing sequence of regular cardinals $\langle\mu\rangle^\smallfrown\langle \kappa_i\mid i<\gamma\rangle$ with $\sup_{i<\gamma}\kappa_i=\kappa$ witnessing  $\mathrm{SLIP}(\mu, \langle \kappa_i\mid i<\gamma\rangle)$.
            \item[$(2)$] For each $\alpha\in S$, there is a subadditive normal coloring $e\colon [\kappa^+]^2\rightarrow \cf(\gamma)$ such that, for each set-sized forcing poset $\mathbb{P}$,  
    $$\forces_{\mathbb{P}}\text{$``\alpha$ is $d$-approachable $\leftrightarrow$ $\kappa^+$ is $e$-approachable}".$$
        \end{enumerate}
    \end{myclaim}
    \begin{proof}[Proof of claim]
    Let us consider the following set
    $$E=\{\rho<\delta^+\mid (1)_{\rho}\,\wedge\,(2)_{\rho}\,\wedge\, (3)_{\rho} \}$$
    where
    \begin{enumerate}
        \item[$(1)_{\rho}$:] $\cof(\rho)<\lambda$ and $\cof(\rho)$ is a successor of a singular cardinal of cofinality $\cf(\gamma)$.
        \item[$(2)_{\rho}$:]  There is an increasing sequence of regular cardinals $\langle \rho_i\mid i<\gamma\rangle$ converging to  $\cof(\rho)^-$,  $\mu<\rho_0$, and $\SLIP(\mu, \langle \rho_i\mid i<\gamma\rangle)$ holds.
        \item[$(3)_{\rho}$:] There is a normal subadditive coloring $e\colon [\cof(\rho)]^2\rightarrow \cf(\gamma)$ such that for each set-sized forcing $\mathbb{P}$, $\mathbb{P}$ forces $``\text{$\rho$ is $d$-approachable $\Leftrightarrow$ $\cof(\rho)$ is $e$-approachable}"$
    \end{enumerate}
    We claim that $E$ is stationary. 
    
    To show this, fix $j\colon V\rightarrow M$ be a $\delta^+$-supercompact embedding and let $C\subseteq\delta^+$ be an arbitrary club. We show that $C$ intersects $E$. Set $\rho:=\sup(j[\delta^+])\in j(C)$ and note that the following properties hold  on the \emph{$M$-side} of $j$:
    \begin{enumerate}
    
        \item Clearly, $\cf(\rho)=\delta^+$. So, $\cf(\rho)\in(j(\mu),j(\lambda))$ and it is the successor of a singular with cofinality $\cf(\gamma)$.
        \item The predecessor of $\cf(\rho)$ is the limit of the increasing sequence $\langle \delta_i\mid i<\gamma\rangle$  of regular cardinals,  $\mu<\delta_0$, and $\mathrm{SLIP}(\mu,\langle \delta_i\mid i<\gamma\rangle)$ holds.
        \item The coloring $e\colon [\cf(\rho)]^2\rightarrow \cf(\gamma)$ defined by
        $$e(\alpha,\beta):=j(d)(j(\alpha),j(\beta))$$
        is normal and subadditive. In addition for each set-sized poset $\mathbb{P}\in M$
        $$\forces_{\mathbb{P}}\text{``$\rho$ is $j(d)$-approachable $\Leftrightarrow$ $\cf(\rho)$ is $e$-approachable''}$$

    \end{enumerate}
    Clauses (1) and (2) are clear, so we just check (3): First $e$ is subadditive just because so is $j(d)$. For normality, letting  $\alpha<\cf(\rho)$ and $i<\cf(\gamma)$ note that 
    $$|\{\beta<\alpha\mid e(\beta, \alpha)\leq i\}|=|\{\beta<\alpha\mid j(d)(j(\beta), j(\alpha))\leq i\}|=|\{\beta<\alpha\mid d(\beta, \alpha)\leq i\}|,$$
    and the latter set has cardinality ${<}\delta$, by normality of $d.$

    Let $\mathbb{P}\in M$ be any set-sized poset and work in $V[\mathbb{P}].$ Then,  $\rho$ is $j(d)$-approachable if and only if $j(d)$ is bounded on an unbounded subset of $j[\delta^+]$. (The left-to-right implication is Fact \ref{fact: Approachability refinement}, and the converse is just trivial.) In turn, $j(d)$ is bounded on an unbounded subset of $j[\delta^+]$ if and only if 
    $e$ is bounded on an unbounded subset of $\delta^+$ by the very definition of $e$. More formally,  if $j(d)$ is bounded on an unbounded set $A\subseteq j[\delta^+]$ with value $i$, $e$ is bounded on $j^{-1}[A]\subseteq\delta^+$ unbounded with value $i$. If $e$ is bounded on an unbounded  set $A\subseteq\delta^+$ with value $i$, $j(d)$ is bounded on $j[A]\subseteq j[\delta^+]$ unbounded with value $i$. 

    \smallskip

    The above argument shows that $\rho\in j(E)\cap j(C)$. Thus,  $E$ is stationary. By Fodor's Lemma, there is $S\s E$ stationary where the map $\rho\mapsto \cof(\rho)$ takes a constant value $<{\lambda}$. This constant value  is, by definition of $E$, a successor of a singular cardinal $\kappa<\lambda$. Clearly, $S$ and $\kappa$ witness the required properties.
    \end{proof}
    Let $\vec{\mathcal{I}},\vec{\mathcal{B}}$ witnesseses for $\mathrm{SLIP}(\mu,\langle \kappa_i\mid i<\gamma\rangle)$ and $G\ast \dot{C}\s \mathbb{P}(\mu,\vec\kappa,\vec{\mathcal{I}},\vec{\mathcal{B}})\ast \dot{\Coll}(\check{\mu},\check{\kappa}^+)$ generic. This forcing does not disturb the GCH pattern of $V$, so Fact~\ref{fact: eisworth handbook} still applies in $V[G\ast \dot{C}]$. With this in mind we prove the following:
\begin{myclaim}
    $V[G\ast \dot{C}]\models ``\delta^+\cap \cof(\mu)\notin I[\delta^+]"$.
    \end{myclaim}
\begin{proof}[Proof of claim]
Suppose towards a contradiction that the claim was false. In the generic extension, $S$ is a stationary subset of $\delta^+\cap \cof(\mu)$ because the forcing yielding $V[G\ast \dot{C}]$ is small relative to $\delta^+$. We next show that $S\notin I[\delta^+]$ holds in $V[G\ast \dot{C}]$. 

First, $d\colon [\delta^+]^2\rightarrow \cf(\gamma)$ is a normal subadditive coloring in $V[G\ast \dot{C}]$ in that this property is generic absolute. As a result, Fact~\ref{fact: eisworth handbook} applies and we will be done after showing that $S\cap S(d)=\emptyset$ holds in $V[G\ast \dot{C}]$. By our prior analysis, 
a point $\alpha\in S$ is $d$-approachable in $V[G\ast \dot{C}]$ if and only if there is a normal coloring $e\colon [\kappa^+]^2\rightarrow \cf(\gamma)$ in $V$ such that $\kappa^+$ is $e$-approachable in $V[G\ast \dot{C}]$. Invoking   Corollary~\ref{Corollary: Top cardinal not approachable} for $\theta:=\mu$ we conclude that this latter condition is false and therefore that no $\alpha\in S$ is $d$-approachable in $V[G\ast \dot{C}]$.
\end{proof}
The above completes the proof of the theorem.  
\end{proof}
\begin{mycol}
    Assume the $\mathrm{GCH}$ holds. Let $\delta$ be a singular limit of supercompact cardinals. For each regular $\mu \in (\cf(\delta),\delta)$ there is a generic extension of the set-theoretic universe where $\delta^+\cap \cof(\mu)\notin I[\delta^+]$. \qed
\end{mycol}
The  next corollary generalizes the main result of \cite{JakobLevine}, where the authors solved Shelah's Question~\ref{que: Shelah} for the case $\kappa=\aleph_\omega$.
\begin{mycol}\label{cor: Failure of Approachability at the aleph}
    Assume the $\mathrm{GCH}$ holds. Let $\lambda$ be a supercompact cardinal, $\aleph_\gamma$ a singular cardinal and  $\langle \delta_i\mid i<\gamma\rangle$  an increasing  sequence of measurable cardinals above $\lambda$ with limit $\delta$. Let $\mu\in(\cf(\gamma),\aleph_\gamma)$ be regular. Then, there is a generic extension where $\aleph_{\gamma+1}\cap \cof(\mu)\notin I[\aleph_{\gamma+1}]$.
\end{mycol}
\begin{proof}
Let $V$ denote the generic extension constructed in Theorem~\ref{theorem: One cofinality} and $S$ the stationary set $\lambda^{+\gamma+1}\cap \cof(\mu)$. Force with the Levy collapse $\mathbb{C} := \Coll(\mu^+, \lambda)$. In the generic extension $V[\mathbb{C}]$, we have $\lambda^{+\gamma+1} = \aleph_{\gamma+1}$, $S$ is stationary and we claim that $S \cap S(d)=\emptyset$ still holds. Assume for a contradiction that $S \cap  S(d)^{V[\mathbb{C}]}$ is non-empty.

Let $\alpha \in S\cap  S(d)^{V[\mathbb{C}]}$ be arbitrary. By Fact \ref{fact: Approachability refinement}, for every unbounded subset $A \subseteq \alpha$ with $\otp(A) = \mu$, there exists an unbounded subset $B \subseteq A$ on which $d$ is bounded. Since $\otp(B) = \mu$ and $\mathbb{C}$ is ${<}\mu^+$-closed, it follows that $B \in V$. This implies that $\alpha \in S\cap S(d)^V$, contradicting our assumption in the ground model. Hence, $S\cap S(d)^{V[\mathbb{C}]}$ is empty and thus, by Fact~\ref{fact: eisworth handbook}, $S\notin I[\aleph_{\gamma+1}].$
\end{proof}

\subsection{On a model of Gitik and Sharon}

In \cite{GitikSharon}, Gitik–Sharon constructed a model in which $\aleph_{\omega^2}$ is strong limit, and both the Singular Cardinal Hypothesis ($\SCH$) and $\AP$ fail at $\aleph_{\omega^2}$. It turns out that in the model of \cite{GitikSharon}, one cannot predetermine the cofinality of the non-approachable stationary set $S \subseteq \aleph_{\omega^2+1}$.  The obstacle, as usual, is that this cofinality is determined by $\cf(\sup(M \cap \kappa^{+\omega+1}))$, where $M \prec H(\Theta)$ is a $\kappa$-Magidor model\footnote{In \cite{GitikSharon}, the authors refer to these models as \emph{supercompact models}. The terminology \emph{$\kappa$-Magidor model} is borrowed from \cite{MohamVelicGuessingModelsApproach}.} and $\kappa$ is a $\kappa^{+\omega+2}$-supercompact cardinal.

In the next theorem, we generalize Gitik–Sharon's result by showing that the failure of $\SCH$ is consistent with the existence of a non-approachable stationary set of any prescribed cofinality. More precisely, we show the following:

\begin{mysen}
    Assume the $\mathrm{GCH}$ holds. Let $\lambda$ be a supercompact cardinal, $\gamma<\lambda$  a limit ordinal and assume that there is an increasing sequence of length $\gamma$ above $\lambda$ consisting  of measurable cardinals. Then, for each regular cardinal $\mu\in (\cf(\gamma),\aleph_{\gamma^2})$, there is a generic extension where $\aleph_{\gamma^2+1}\cap \cof(\mu)\notin I[\aleph_{\gamma^2+1}]$ and $\SCH_{\aleph_{\gamma^2}}$ fails. 
\end{mysen}
\begin{proof}
   Let $d\colon [\delta^+]^2 \rightarrow \gamma$ be a normal subadditive coloring, where $\delta$ is the supremum of the aforementioned sequence of measurable cardinals. By the argument in Theorem~\ref{theorem: One cofinality}, there is a generic extension in which $\lambda$ remains supercompact and there exists a stationary set $S \subseteq \lambda^{+\gamma+1} \cap \operatorname{cf}(\mu)$ consisting of non-$d$-approachable points. Working over this generic extension we first make $\lambda$ indestructible under $\lambda$-directed-closed forcings (see \cite{LaverIndestruct}), and then force with Cohen forcing $\operatorname{Add}(\lambda, \lambda^{+\omega+2})$. In the resulting extension $\lambda$ remains supercompact. We can arrange this latter two-step iteration to be ${<}\mu^+$-directed-closed and $\lambda^+$-cc. As a result,  $S$ remains stationary and no new $d$-approachable points of cofinality $\mu$ are added. Thus, $S \cap S(d) = \emptyset$ also holds in this model. For simplicity, denote this  extension by $V$.

    \smallskip

 Building upon work of Gitik–Sharon, Sinapova showed in her Ph.D. thesis \cite{SinapovaPhD} that in $V$ there exists a Mitchell-increasing sequence $\vec{\mathcal{U}} = \langle \mathcal{U}_\xi \mid \xi < \gamma \rangle$ of normal fine ultrafilters on $\mathcal{P}_\lambda(\lambda^{+\xi})$, together with a sequence of \emph{guiding generics} $\vec{K} = \langle K_\xi \mid \xi < \gamma \rangle$ for the Levy collapse $\Coll(\lambda^{+\gamma+2}, {<}j_{\mathcal{U}_\xi}(\lambda))^{\mathrm{Ult}(V, \mathcal{U}_\xi)}$. Using this, she defined a Magidor-style variant of the Gitik–Sharon forcing, denoted $\mathbb{S}(\vec{\mathcal{U}}, \vec{K})$.
 
 \smallskip

 $\mathbb{S}(\vec{\mathcal{U}}, \vec{K})$ satisfies the following properties:

  \begin{enumerate}
      \item $\mathbb{S}(\vec{\mathcal{U}}, \vec{K})$ is $\lambda^{+\gamma+1}$-cc. (See \cite[\S3]{SinapovaPhD}).
      \item  $\mathbb{S}(\vec{\mathcal{U}}, \vec{K})$  collapses all cardinals in the interval $(\lambda, \lambda^{+\gamma+1}).$ (See \cite[\S3]{SinapovaPhD})
      \item $\mathbb{S}(\vec{\mathcal{U}}, \vec{K})$ forces $``\lambda=\dot{\kappa}_0^{+\gamma^2}$", where $\dot{\kappa}_0$ is the name for the first \emph{Prikry point}.
      \item  $\langle \mathbb{S}(\vec{\mathcal{U}}, \vec{K}),\leq, \leq_0\rangle$ has the Strong Prikry Property (See \cite[Lemma~3.18]{PovTree}.)\footnote{Technically, the proof of \cite[Lemma~3.18]{PovTree} is done for Sinapova's forcing without interleaved collapses, but a similar proof goes through even if one interleaves collapses.}
      \item There is a dense subposet of $\mathbb{S}(\vec{\mathcal{U}}, \vec{K})$ whose $\leq_0$-ordering is ${<}\mu^+$-closed.\footnote{To show this, one just forces below a condition whose measure one sets concentrate on sets $x$ whose associated Prikry points are above $\mu$. }
  \end{enumerate}
  To simplify notations, let us denote the above dense subposet  $\mathbb{S}$. 

  \begin{myclaim}
      $\forces_{\mathbb{S}} ``\check{S}\s \lambda^{+\gamma+1}\cap\cof(\mu)$ is stationary and $\check{S}\cap \dot{S}(\check{d})=\emptyset".$
  \end{myclaim}
  \begin{proof}[Proof of claim]
     Clearly, $S$ remains stationary by the chain condition of the forcing. The reason why all points in $S$ retain cofinality $\mu$ is essentially the same as that explaining why $S \cap S(d)^{V[S]} = \emptyset$. Thus, we provide details only for the latter part.

     \smallskip

It suffices to prove that for $\alpha<\lambda^{+\gamma+1}$ with $\cf(\alpha)=\mu$, the truth value of ``$\alpha$ is not $d$-approachable'' is not changed by forcing with $\dS(\vec{\mathcal{U}},\vec{K})$. So let $\alpha<\lambda^{+\gamma+1}$ have cofinality $\mu$. Let $A\subseteq\alpha$ be unbounded with ordertype $\mu$. By Fact \ref{fact: Approachability refinement}, after forcing with $\dS(\vec{\mathcal{U}},\vec{K})$, $\alpha$ is $d$-approachable if and only if $d$ is bounded on an unbounded subset of $A$. However, since $A$ has size $\mu$ in the ground model and $\dS(\vec{\mathcal{U}},\vec{K})$ does not add any new subsets of $\mu$ by point (5), $\alpha$ clearly remains non-$d$-approachable after forcing with $\dS(\vec{\mathcal{U}},\vec{K})$.
  \end{proof}
Let $G\s \mathbb{S}$ be a generic filter. In $V[G]$, $\mu$ is a regular cardinal, $\lambda=(\kappa_0^{+\gamma^2})^V$ and $\lambda^+=(\lambda^{+\gamma+1})^V$. (Here $\kappa_0=(\dot{\kappa}_0)_G$.) Now, force with $\Coll(\mu^+,{<}\kappa_0)$. Let $H$ be a generic filter for this poset. In $V[G\ast H]$,  $\kappa_0$ becomes $\mu^+$ and  $\lambda=\aleph_{\gamma^2}$ because orginally $\mu<(\aleph_{\gamma^2})^V$. In addition, this forcing does preserve the stationarity of $S$ and does not create new $d$-approachable points of cofinality $\mu$. All in all, $S$ does not intersect $S(d)^{V[G\ast H]}$, which yields $\aleph_{\gamma^2+1}\cap \cof(\mu)\notin I[\aleph_{\gamma^2+1}]$. In $V[G\ast H]$, $\aleph_{\gamma^2}$ is a strong limit cardinal at which the SCH fails. 
\end{proof}

\section{Failure of Approachability at every Singular Cardinal}\label{sec: failure everywhere}

The main result of this section is Theorem~\ref{ThmD}. Namely:

\begin{mysen}\label{theorem: AP Fails at every singular}
    Suppose that there is a class of supercompact cardinals. Then, there exists a class forcing extension in which $\AP_{\kappa}$ fails for every singular cardinal $\kappa$.

     Moreover, for each singular cardinal $\kappa$, there are unboundedly many regular cardinals $\delta<\kappa$ for which $\kappa^+\cap\cof(\delta)\notin I[\kappa^+]$.
\end{mysen}
From the above we immediately deduce the following corollary:
\begin{mycol}
  Suppose that there is a class of supercompact cardinals. Then, there exists a model of $\ZFC$ without special $\kappa^+$-Aronszajn trees any singular cardinal $\kappa$.   
\end{mycol}

The proof of this theorem consists of two parts: Given a regular cardinal $\mu$ and a supercompact cardinal $\kappa>\mu$, first we construct a forcing $\mathbb{A}(\mu,\kappa)$ that collapses $\kappa$ to become $\mu^+$ which when followed with a class product of Levy collapses $\mathbb{C}$, the two-step iteration $\mathbb{A}(\mu,\kappa)\ast \dot{\mathbb{C}}$ forces $\delta^+\cap \cof(\mu)\notin I[\delta^+]$ for every singular cardinal $\delta>\kappa$ with $\cf(\delta)<\mu$. In the second part we show how to iterate this construction to produce a model where $\AP_\delta$ fails for every singular $\delta.$

\smallskip

Let us begin defining the poset $\mathbb{A}(\mu,\kappa)$. 
Recall that if $\kappa$ is a supercompact cardinal, $\ell\colon \kappa\rightarrow V_\kappa$ is called a \emph{Laver function} if for each set $x$ and a regular cardinal $\Theta\geq \kappa$ such that $x\in H(\Theta)$ there is an elementary embedding $j\colon V\rightarrow M$ with critical point $\kappa$, $j(\kappa)>\Theta$, $M$ is closed under $\Theta$-sequences of its members and $j(\ell)(\kappa)=x$. In \cite{LaverIndestruct}, Laver showed that Laver functions always exist. This makes sense of the next definition:
\begin{mydef}[The poset $\mathbb{A}(\kappa,\mu)$]\label{def: The A poset}
    Let $\mu$ be regular and $\kappa>\mu$ supercompact. Fix a Laver function $\ell\colon\kappa\to V_{\kappa}$. We define a poset $\dA(\mu,\kappa)$ as the limit of an Easton-support Magidor iteration  of Prikry-type forcings $(\mathbb{P}_\alpha, \dot{\mathbb{Q}}_\alpha)_{\alpha<\kappa}$ with the following iterands: Suppose we have defined $(\dP_{\beta},\dot{\dQ}_{\beta})_{\beta<\alpha}$. We let $\dP_{\alpha}$ be the Easton-support Magidor limit of $(\dP_{\beta},\dot{\dQ}_{\beta})_{\beta<\alpha}$. We let $\dot{\dQ}_{\alpha}$ be defined as follows:
    \begin{itemize}
        \item \textbf{Case 1:} Suppose that $\alpha$ is inaccessible, $|\dP_{\beta}|<\alpha$ for every $\beta<\alpha$ and $$\ell(\alpha)=(\dot{\dL}_{\alpha}, \langle \delta_i\mid i<\gamma\rangle,\dot{\vec{\mathcal{I}}}, \dot{\vec{\mathcal{B}}})$$ where
        \begin{itemize}
            \item   $\dot{\dL}_{\alpha}$ is a $\dP_{\alpha}$-name for a ${<}\,\alpha$-strategically closed poset.
            \item $\forces_{\mathbb{P}_\alpha\ast \dot{\mathbb{L}}_\alpha}``\SLIP(\check{\mu},\langle \check{\delta}_i\mid i<\check{\gamma}\rangle)$ holds     as witnessed by $\dot{\vec{\mathcal{I}}}$ and $\dot{\vec{\mathcal{B}}}$".
        \end{itemize}

         In this case, we let $\delta_{\alpha}^*:=(\sup_{i<\gamma}\delta_i)^+$ and let $\dot{\dQ}_{\alpha}$ be a $\dP_{\alpha}$-name for
        $$\dot{\dL}_{\alpha}*\dot{\dP}(\check{\mu},\check{\vec{\delta}},\dot{\vec{\mathcal{I}}},\dot{\vec{\mathcal{B}}})*\dot{\Coll}(\check{\mu},\check{\delta}_{\alpha}^*).$$
        \item \textbf{Case 2:} Otherwise, we let $\dot{\dQ}_{\alpha}$ be a $\dP_{\alpha}$-name for $\dot{\Coll}(\check{\mu},|\check{\dP}_{\alpha}|)$.
    \end{itemize}
\end{mydef}

We first justify the following easy properties of $\dA(\mu,\kappa)$:
\begin{mylem}\label{lemma: Properties of A}
    Let $\mu$ be regular and $\kappa>\mu$ supercompact.
    \begin{enumerate}
        \item $(\dA(\mu,\kappa),\leq,\leq_0)$ is a Prikry-type forcing and in fact with pure $\mu$-decidability.
        \item $(\dA(\mu,\kappa),\leq)$ and $(\dA(\mu,\kappa),\leq_0)$ are $\kappa$-cc.
        \item $(\dA(\mu,\kappa),\leq_0)$ is ${<}\,\mu$-directed closed.
        \item $(\mathbb{A}(\mu,\kappa), \leq)$ forces $``\mu^+=\kappa".$
    \end{enumerate}
\end{mylem}

\begin{proof}
    Point (1) follows from Lemma \ref{lemma: Easton support Magidor}, together with Lemma \ref{lemma: P has Prikry prop}, noticing that we have assumed $\delta_0>\mu$ and thus that the intermediate forcing has pure $\mu$-decidability by virtue of Lemma~\ref{lemma: P has Prikry prop}.  Both parts of point (2) follow easily using a $\Delta$-system argument. Point (3) follows from Lemma \ref{lemma: Iteration pure closure}. Point (4) follows from standard genericity considerations.
\end{proof}

The following result is instrumental for our main theorem. Essentially, it says that after forcing with $\mathbb{A}(\mu,\kappa)$ not only $\{\alpha<\delta^+\mid \text{$\alpha$ is not $d$-approachable}\}$ is stationary, but also it remains stationary (and hence, non-approachable) in the generic extension by any ${<}\kappa$-directed closed forcing $\mathbb{P}.$

\begin{mysen}[Indestructibility of $\neg \AP$]\label{Theorem: Indestructibility of neg AP}
    Let $\mu$ be regular and $\kappa>\mu$ supercompact. Let $\gamma<\mu$ be a limit ordinal and $\vec{\delta}=\langle \delta_i\mid i<\gamma\rangle$ an increasing and continuous sequence of regular cardinals with $\delta_0\geq\kappa$. Let $\delta:=\sup_{i<\gamma}\delta_i$ and $d\colon[\delta^+]^2\to\cf(\gamma)$ be a normal, subadditive coloring.

    Let $G$ be $\dA(\mu,\kappa)$-generic and work in $V[G]$. Whenever $\dP$ is a ${<}\,\kappa$-directed closed poset which preserves $\delta^+$, every $\delta_i$ and forces $\SLIP(\mu,\vec{\delta})$, for any large enough regular cardinal $\Theta$ there are stationarily many $M\in[H^{V[G]}(\Theta)]^{<\kappa}$ such that:
    \begin{enumerate}
        \item $M\prec H^{V[G]}(\Theta)$.
        \item $\dA(\mu,\kappa),\dP,d\in M$.
        \item $\cf(\sup(M\cap\delta^+))=\mu$ and $\sup(M\cap\delta^+)$ is not $d$-approachable.
        \item \textbf{(Existence of master conditions)} Whenever $p\in\dP\cap M$ there is a condition $p^*\in\dP$  with $p^*\leq p$ such that whenever $D\in M$ is open dense in $\dP$, there is $q\in D\cap M$ with $p^*\leq q$.
    \end{enumerate}
\end{mysen}

\begin{proof}
    Fix parameters as in the theorem and assume that they belong to $H(\Theta)$. In $V[G]$, let $F\colon[H^{V[G]}(\Theta)]^{<\omega}\to[H^{V[G]}(\Theta)]^{<\kappa}$ be any function. We want to find $M\in[H^{V[G]}(\Theta)]^{<\kappa}$ which is closed under $F$ and satisfies the additional requirements (1)--(3). To this end, let $\dot{F}$ and $\dot{\dP}$ be $\dA(\mu,\kappa)$-names with $\dot{F}^G=\dot{F}$ and $\dot{\dP}^G=\dP$. Assume without loss of generality that the empty condition of $\mathbb{A}(\mu,\kappa)$ forces $\dot{\dP}$ to have the claimed properties. Let $\dot{\mathcal{I}}_i,\dot{\mathcal{B}}_i$ be $\dP$-names forced to witness $\SLIP(\check{\mu},\check{\vec{\delta}})$ and  $\Theta'>\Theta$ be a big enough regular cardinal such that $\dot{F},\dot{\dP}, \dot{\mathcal{I}}_i,\dot{\mathcal{B}}_i\in H^V(\Theta')$. 

    \smallskip

    Let $j\colon V\to M$ be a $|H^V(\Theta')|$-supercompact embedding for $\kappa$ such that $$j(\ell)(\kappa)=(\dot{\dP}, \langle \delta_i\mid {i<\gamma}\rangle,\dot{\vec{\mathcal{I}}},\dot{\vec{\mathcal{B}}}).$$ 
    It follows that in $M$, we have
    $$j(\dA(\mu,\kappa))=\dA(\mu,\kappa)*(\dot{\dP}*\dot{\dP}(\check{\mu},\check{\vec{\delta}},\dot{\vec{\mathcal{I}}},\dot{\vec{\mathcal{B}}})*\dot{\Coll}(\check{\mu},\check{\delta}^+))*j(\dot{\dA})_{\kappa+1,j(\kappa)}$$
    where $j(\dot{\dA})_{\kappa+1,j(\kappa)}$ is forced to be an Easton support Magidor iteration of Prikry-type forcings with pure $\check{\mu}$-decidability and a ${<}\,\check{\mu}$-closed pure extension ordering. 

    \smallskip

    Let $H$ be $j(\dA(\mu,\kappa))$-generic over $V$ such that $j[G]\s H$. Then $j$ lifts to $$j\colon V[G]\to M[H].$$ 
    Let $N':=j[H^V(\Theta')][H]$ and $N:=N'\cap H^{M[H]}(j(\Theta))$. We want to show that in $M[H]$,  $N$ is as required by the theorem with $j$ applied to the parameters.

    \setcounter{myclaim}{0}

    \begin{myclaim}
        $j(\mathbb{A}(\mu,\kappa)), j(\mathbb{P}), j(d)\in N$.
    \end{myclaim}
    \begin{proof}[Proof of claim]
       This is clear by our choice of $\Theta'.$ 
    \end{proof}

    \begin{myclaim}
        $N\prec H^{M[H]}(j(\Theta))$, $|N|<j(\kappa)$ and $j(F)``[N]^{<\omega}\s N.$
    \end{myclaim}
    \begin{proof}[Proof of claim]
       Clearly, $|N|<j(\kappa)$ as $j(\kappa)>\Theta'=|N|.$ It is also  routine to check that $N'\prec H^{M[H]}(j(\Theta'))$. Using this one can show that $N$ is closed under $j(F)$:  Given any $(x_0,\dots, x_k)\in [N]^{<\omega}$ there are $(\tau_0,\dots, \tau_n)\in H^V(\Theta')$ such that $x_i=j(\tau_i)^H.$ By definition, $j(\dot{F})^H(j(\tau_0)^H,\dots, j(\tau_k)^H)=j(\dot{F}^G(\tau_0^G,\dots, \tau_k^G))\in j[H^{V[G]}(\Theta)]\s H^{M[H]}(j(\Theta')).$ In addition, this very same value of $j(F)$ belongs to $N':=j[H^V(\Theta')][H]$ as this is an elementary submodel of $H^{M[H]}(j(\Theta')).$ 

       Finally, we argue that $N\prec H^{M[H]}(j(\Theta))$.  Suppose that $H^{M[H]}(j(\Theta))\models ``\exists x\varphi(x,a)"$ for some $a\in N.$ Then, the following holds in the bigger model $H^{M[H]}(j(\Theta'))$:
       \begin{center}
           ``There is $b\in H^{M[H]}(j(\Theta))$ such that $H^{M[H]}(j(\Theta))\models \varphi(b,a)$".
       \end{center}
       This is a first-order formula with parameters $\{a,b,j(\Theta)\}$, all of which available to the model $N'$. Thus, by elementarity, there is $b\in N'\cap H^{M[H]}(j(\Theta))=N$ with the previous property. As a result $N\prec H^{M[H]}(j(\Theta)).$
    \end{proof}

    \begin{myclaim}
        $\cf^{M[H]}(\sup(N\cap j(\delta^+)))=\mu.$
    \end{myclaim}
    \begin{proof}[Proof of claim]
       Since $j(\kappa)>\Theta'$, we have $N'\cap j(\kappa)\in j(\kappa)$.  Ergo, since $H$ is generic for a $j(\kappa)$-cc forcing, it is a standard fact that $N'\cap M=j[H^V(\Theta')]$. In particular, $N\cap j(\delta^+)=N'\cap j(\delta^+)=j[\delta^+]$. Because of $\Coll(\mu,\delta^+)$, $\cf^{M[H]}(\sup j[\delta^+])=\mu$.

       \smallskip
       
       For the reader's benefit we provide the proof of the  standard fact.

       The $\supseteq$-inclusion is clear. For the $\s$-inclusion, let $\tau^H\in N'\cap M$ with $\tau\in j[H^V(\Theta')]$. Since  $j[H^V(\Theta')]\prec j(H^V(\Theta'))$, there is a maximal antichain $A\in j[H^V(\Theta')]$ of conditions forcing $\tau=\check{x}$ for some $x\in j(H^V(\Theta'))$. Since $j(\mathbb{Q})$ is $j(\kappa)$-cc, there is  an enumeration  $f\colon \mu\rightarrow A$ for some cardinal $\mu<j(\kappa)$. By elementarity, $$\mu\in j[H^V(\Theta')]\cap j(\kappa)=N'\cap j(\kappa)=\kappa\s j[H^V(\Theta')].$$
       It follows that $A=f[\mu]\s j[H^V(\Theta')]$. In particular, there is a condition $q\in H\cap j[H^V(\Theta')]$ forcing $``\tau=\check{x}$" for some $x\in j(H^V(\Theta'))$. By elementarity, $x\in j[H^V(\Theta')]$. Thus, $\tau^H=x\in j[H^V(\Theta')]$, as needed.
    \end{proof}
    
    Let us denote $\rho:=\sup(j[\delta^+])=\sup(N\cap j(\delta^+))$. 
    \begin{myclaim}
        $\rho$ is not $j(d)$-approachable in $M[H].$ 
    \end{myclaim}
    \begin{proof}[Proof of claim]
         In $M$, $\cf(\rho)=\delta^+$. $\delta^+$ is regular after forcing with $\dA(\mu,\kappa)*\dot{\dP}$ and is collapsed to have size and cofinality $\mu$ after forcing with $$\dA(\mu,\kappa)*\dot{\dP}*\dot{\dP}(\check{\vec\delta},\dot{\vec{\mathcal{I}}},\dot{\vec{\mathcal{B}}})*\dot{\Coll}(\check{\mu},\check{\delta}^+).$$ This is still the case in $M[H]$ because $\mu$ is preserved. In summary, $$M[H]\models ``\cf(\rho)=\cf(\delta^+)=\mu".$$ So we have to show that $\rho$ is not $j(d)$-approachable in $M[H]$. As in the proof of Theorem \ref{theorem: One cofinality}, in $M$, there is a normal subadditive coloring $e$ on $\delta^+$ such that in any set-sized forcing extension, $\rho$ is not $j(d)$-approachable if and only if $\delta^+$ is not $e$-approachable. After forcing with $\dA(\mu,\kappa)*\dot{\dP}$, by assumption, $\delta^+$ and every $\delta_i$ are cardinals and therefore $e$ is still a normal, subadditive coloring. By Corollary \ref{Corollary: Top cardinal not approachable}, $\delta^+$ is not $e$-approachable after  forcing with $$\dot{\dP}(\check{\vec{\delta}},\dot{\vec{\mathcal{I}}},\dot{\vec{\mathcal{B}}})*\dot{\Coll}(\check{\mu},\check{\delta}^+).$$ Ergo, $\rho$ is not $j(d)$-approachable in $M[H\restriction\kappa+1]$  where $H\uhr\kappa+1$ is the $\dA(\mu,\kappa)*\dot{\dP}*\dot{\dP}(\check{\vec{\delta}},\dot{\vec{\mathcal{I}}},\dot{\vec{\mathcal{B}}})*\dot{\Coll}(\check{\mu},\check{\delta}^+)$-generic filter induced by $H$. 

         \smallskip
    
    Let us show that this is preserved by  the tail forcing $$\dR:=j(\dot{\dA}(\mu,\kappa))_{\kappa+1,j(\kappa)}^{H\uhr\kappa+1}.$$
    
    As stated before, $\dR$ is a Prikry-type poset with pure $\mu$-decidability and a ${<}\,\mu$-closed pure extension ordering. In $M[H\uhr\kappa+1]$, let $\Delta\colon\mu\to\rho$ be increasing and cofinal. In $M[H]$, $\rho$ is $j(d)$-approachable if and only if there is an unbounded subset of $\im(\Delta)$ on which $j(d)$ is bounded (by Fact~\ref{fact: Approachability refinement}). So assume that $\dot{\Gamma}$ is an $\dR$-name for an increasing function from $\check{\mu}$ into $\im(\check{\Delta})$ such that $j(d)$ is bounded on its image (say, with value $\check{\zeta}$), forced by some $r\in\dR$. We construct an increasing sequence $\langle \alpha_i\mid i<\mu\rangle$ of elements of $\mu$ and a $\leq_0$-decreasing sequence $\langle r_i\mid i<\mu\rangle$ of elements of $\dR$ such that for any $i<\mu$, $r_i\Vdash\dot{\Gamma}(\check{i})=\check{\Delta}(\check{\alpha}_i)$. Assume that $\langle \alpha_j\mid j<i\rangle$ and $\langle r_j\mid j<i\rangle$ have been constructed. Let $r_i'$ be a $\leq_0$-lower bound of $\langle r_j\mid{j<i}\rangle$ (or $r_i=r$ if $i=0$). Then $r_i'\Vdash\exists\alpha<\check{\mu}(\dot{\Gamma}(\check{i}))=\check{\Delta}(\check{\alpha})$, so by the pure $\mu$-decidability of $\dR$ we can find $\alpha_i$ and $r_i\leq_0r_i'$ such that
    $$r_i\Vdash\dot{\Gamma}(\check{i})=\check{\Delta}(\check{\alpha}_i)$$
    Since $\dot{\Gamma}$ is forced to be increasing and $\Delta$ is increasing, $\alpha_i>\alpha_j$ for all $j<i$. Now we claim that $\{\Delta(\alpha_i)\;|\;i<\mu\}$ is a cofinal subset of $\rho$ on which $j(d)$ is bounded with value $\zeta$, obtaining a contradiction as that set is in $M[H\uhr\kappa+1]$. First observe that, because the sequence $\langle \alpha_i\mid i<\mu\rangle$ is increasing, the set $\{\alpha_i\;|\;i<\mu\}$ is cofinal in $\mu$ and so $\{\Delta(\alpha_i)\;|\;i<\mu\}$ is cofinal in $\rho$. Lastly, whenever $j<i<\mu$, we have
    $$r_i\Vdash j(\check{d})(\check{\Delta}(\check{\alpha}_j),\check{\Delta}(\check{\alpha}_i))=j(\check{d})(\dot{\Gamma}(\check{j}),\dot{\Gamma}(\check{i}))\leq\check{\zeta}$$
    so actually $j(d)(\Delta(\alpha_j),\Delta(\alpha_i))\leq\zeta$.
    \end{proof}

    \smallskip

    We complete the argument  showing that $N$ admits ``master conditions" for $j(\mathbb{P})$.
    \begin{myclaim}
   Whenever $p\in j(\mathbb{P})\cap N$ there is a condition $p^*\in j(\mathbb{P})$ with $p^*\leq p$ such that whenever $D\in N$ is dense open there is $q\in D\cap N$ such that $p^*\leq q.$
    \end{myclaim}
    \begin{proof}[Proof of claim]
   Fix $j(p)\in j(\mathbb{P})\cap N$.
     By construction, we know that in $M[H]$ there is a filter $K$ which is $\dP$-generic over $M[G]$, which we can assume has $p$ as a member. By our assumptions, we know that $M$ is closed under $|\dP|$-sequences and so $j\uhr\dP$ is a member of $M[G]$. In particular, $j[K]:=(j\uhr\dP)[K]$ is also a member of $M[H]$, since $K\in M[H]$. We have $|K|<\Theta'<j(\kappa)$ and so by the ${<}\,j(\kappa)$-directed closure of $j(\dP)$ in $M[H]$ there is a condition $p^*\in j(\dP)$ which is a lower bound of $j[K]$. 
     We claim that $p^*$ is required. To this end, let $D\in N$ be open dense in $j(\dP)$. By the definition of $N$, $D=j(E)$ for $E\in H^{V[G]}(\Theta')$, $E$ open dense in $\dP$. It follows that there is $q\in E\cap K$. Then $j(q)\in D\cap j[K]\cap N$ and by construction, $p^*\leq j(q)$.   
    \end{proof}

    This finishes the proof.
\end{proof}

We obtain the following generalization to our answer of Shelah's problem:
\begin{mysen}\label{theorem: almost theorem A}
Assume the $\mathrm{GCH}$ and that there is a supercompact cardinal $\kappa$ together with a proper class of measurables. Let $\mu<\kappa$ be regular. Then, there is a model of $\ZFC$ in which $\delta^+\cap \cof(\mu)\notin I[\delta^+]$ for every singular $\delta>\kappa$ with $\cf(\delta)<\mu$.
\end{mysen}

\begin{proof}
    Let $G$ be $\dA(\mu,\kappa)$-generic. Standard arguments show that the GCH pattern survives in $V[G]$. In $V[G]$, let $\langle \kappa_i\mid i\in \On\rangle$ be an increasing and continuous sequence such that $\kappa_0=\kappa$ and $\kappa_i$ is measurable whenever $i$ is a successor. (Note that we allow the existence of measurable cardinals that are limit of measurables.) 
    
    We force with $\mathbb{C}$, the $\On$-length Easton support product of $\Coll(\kappa_i^{++},{<}\kappa_{i+1})$. By the increasing closure of the posets (and Easton's lemma, Fact~\ref{lemma: Eastons lemma}) one can show that $\mathbb{C}$ is \emph{tame} in the sense of Friedman \cite{FriedmanBook}. In particular, 
    all the standard results from forcing theory apply to $\mathbb{C}$ and the resulting generic extension is a model of $\ZFC.$ In the extension in $\mathbb{C}$, all cardinals below $\kappa=\mu^+$ survive. Standard arguments show that above $\kappa$ all the $V$-cardinals of the form $\kappa_i$, $\kappa_i^+$ and $\kappa_i^{++}$ remain cardinals while the rest are collapsed. 
    It also follows that the only singular cardinals above $\kappa$ in the extension are  the $\kappa_i$'s that were formerly singular (see Lemma~\ref{lemma: singular cardinals}).

   Let  $C\s \mathbb{C}$ be $V[G]$-generic. For each  $i\in \On$ denote $$\textstyle \mathbb{C}^i:=\prod_{j\geq i}\Coll(\kappa_j^{++},{<}\kappa_{j+1})\;\text{and}\; \mathbb{C}_i:=\prod_{j<i}\Coll(\kappa_j^{++},{<}\kappa_{j+1}).$$
   Working in $V[G\ast C]$, let  $\delta>\kappa$ be a singular cardinal with $\cf^{V[C]}(\delta)<\mu$. Then, there exists a limit ordinal $\gamma$  such that $\delta=\kappa_\gamma$. 

   \smallskip

   Note that $\delta^+\cap \cof(\mu)\notin I[\delta^+]$ holds in $V[G\ast  C]$ if and only if $\delta^+\cap \cof(\mu)\notin I[\delta^+]$ holds in $V[C\cap \mathbb{C}_i]$ because the tail forcing $\mathbb{C}^\gamma$ is $\kappa_{\gamma}^{++}$-distributive in $V[C\cap \mathbb{C}_\gamma]$ (by Easton's lemma). Thus, it will be sufficient to check this latter property.

   \smallskip

   For each  $i<\gamma$, $\kappa_i$ is a measurable cardinal in $V[G]$ and so $\Coll(\kappa_i^{++},{<}\kappa_{i+1})$ is a $\LIP(\kappa_{i}^{++},\kappa_i)$-forcing. Combining this with   Lemma \ref{lemma: Product forces LIP} we infer that $\mathbb{C}_\gamma$ forces $\SLIP(\kappa_0,\langle\kappa_{i}\mid i<\gamma\rangle)$. So $\mathbb{C}_\gamma$ satisfies the requirements of Theorem~\ref{Theorem: Indestructibility of neg AP}.     Let $d\colon[\kappa_i^+]^2\to\cf(\gamma)$ be a normal, subadditive coloring in $V$ and denote $S:=\delta^+\cap \cof(\mu)$ as computed in $V[G]$. Note that the definition of this set is absolute between $V[G]$ and $V[G\ast C]$. By Fact \ref{fact: eisworth handbook}, it suffices to show that $S$ remains stationary after forcing with $\dC_\gamma$ and that it consists of points that are not $d$-approachable.

    \begin{myclaim}
       Working over $V[G]$, $\mathbb{C}_\gamma$ forces that $\check{S}$ is stationary and $\check{S}\cap \dot{S}(\check{d})=\check{\emptyset}.$
    \end{myclaim}
    \begin{proof}[Proof of claim]
Work in $V[G]$. We perform a density argument. Suppose that $\dot{C}$ is a $\dC_\gamma$-name for a club in $\delta^+$ as forced by a condition $p\in \mathbb{C}_\gamma$.    By Theorem \ref{Theorem: Indestructibility of neg AP}, there is $M\prec H^{V[G]}(\Theta)$ and $p\in \mathbb{C}_\gamma$ such that  $\rho:=\sup(M\cap\delta^+)$ 
    non $d$-approachable, $\cf(\rho)=\mu$, and $p^*$ is a $(\mathbb{C}_\gamma, M)$-master condition with $p^*\leq p$. Recall that this means that 
whenever $D\in M$ is open dense in $\dC_\gamma$, there is $q\in D\cap M$ with $p^*\leq q$. 

We claim that $p^*$ forces that $\check{\rho}\in\dot{C}\cap \check{S}$. To this end, it suffices to show that $p^*$ forces $\dot{C}\cap\check{\rho}$ to be unbounded in $\check{\rho}$. Given any $\zeta<\rho$, let $\zeta'>\zeta$ be with $\zeta'\in M$. The open dense set of all conditions in $\dC_\gamma$ forcing $\check{\eta}\in\dot{C}$ for some $\eta>\zeta'$ is in $M$. As a result, there is a condition $q\in D\cap M$ such that $p^*\leq q$. By elementarity, there is a witnessing $\eta$ in $M$. Therefore, $p^*$ forces $\dot{C}$ to be unbounded in $\check{\rho}$.  

Lastly, $p^*$ also forces $``\rho$ is  not $d$-approachable"  because $\dC_\gamma$ adds no new subsets of $\mu$ and $\cf(\rho)=\mu$. (So, in fact, the weakest condition forces this property). 

By density it follows that 
$\forces_{\mathbb{C}_\gamma}\text{$``\check{S}$ is stationary and $\check{S}\cap \dot{S}(\check{d})= \check{\emptyset}$"}.$
    \end{proof}
We are done with the theorem.
\end{proof}

Now we iterate the construction. To obtain enough instances of $\LIP$ along the way we alternate the forcing $\dA$ with the Levy-collapse. Additionally, in order to show that there is a forcing which projects onto our tail and forces $\SLIP$, we leave a gap of two cardinals between the collapses (so that Lemma \ref{lemma: Product forces LIP} applies). This necessarily leads to gaps between the cofinalities at which $\AP$ fails. We do not know if this is necessary, see Question \ref{Question 1}.

\begin{mydef}
    Assume $\GCH$. Let $\langle\kappa_\alpha\mid \alpha\in \On\rangle$ be an increasing and continuous sequence of cardinals such that $\kappa_0=\aleph_0$ and $\kappa_{\alpha}$ is supercompact whenever $\alpha$ is a successor. We define an Easton support Magidor iteration $$((\dP_{\alpha},\leq_{\alpha},\leq_{\alpha,0}),(\dot{\dQ}_{\alpha},\dot{\leq}_{\alpha},\dot{\leq}_{\alpha,0}))_{\alpha\in\On}$$ as follows: Assume that $(\dP_{\beta},\dot{\dQ}_{\beta})_{\beta<\alpha}$ has been defined. Let $\dP_{\alpha}$ be the Easton limit.
    \begin{itemize}
        \item \textbf{Case 1:} $\alpha$ is even. We let $\dot{\dQ}_{\alpha}$ be a $\dP_{\alpha}$-name for $\dot{\Coll}(\kappa_{\alpha}^{++},<\kappa_{\alpha+1})$
        \item \textbf{Case 2:} $\alpha$ is odd. We let $\dot{\dQ}_{\alpha}$ be a $\dP_{\alpha}$-name for $\dot{\dA}(\kappa_{\alpha}^{++},\kappa_{\alpha+1})$.
    \end{itemize}
\end{mydef}

\begin{mydef}
    Let $\beta<\alpha$ be ordinals. We let $(\dot{\dP}_{\beta,\alpha},\dot{\leq}_{\beta,\alpha},\dot{\leq}_{(\beta,\alpha),0})$ be a $\dP_{\beta}$-name for a Prikry-type poset such that $\dP_{\beta}*\dot{\dP}_{\beta,\alpha}\cong\dP_{\alpha}$.
\end{mydef}

So, in $V[\dP_{\beta}]$, $\dP_{\beta,\alpha}$ is the Easton support Magidor iteration of $(\dot{\dP}_{\beta,\gamma},\dot{\dQ}_{\gamma})_{\gamma\in[\beta,\alpha)}$. 

\begin{mylem}\label{lemma: Tail properties}
    Let $\beta<\alpha$ be ordinals. $\dP_{\beta}$ forces that $\dot{\dP}_{\beta,\alpha}$ is a Prikry-type poset with pure $\kappa_{\beta}^{++}$-decidability and a ${<}\,\kappa_{\beta}^{++}$-closed pure extension ordering. Additionally, if $\alpha$ is a limit, then $\dot{\dP}_{\beta,\alpha}$ is forced to have almost pure $\kappa_{\alpha}^+$-decidability.
\end{mylem}

\begin{proof}
    The decidability of $\dot{\dP}_{\beta,\alpha}$ follows from Lemma \ref{lemma: Easton support Magidor} together with Lemma \ref{lemma: Properties of A}. Lemma \ref{lemma: Easton support Magidor} also implies the almost pure ${<}\,\kappa_{\alpha}^+$-decidability, as every $\dot{\dP}_{\beta,\gamma}$ has it for $\gamma<\alpha$ since it is of size ${<}\,\kappa_{\alpha}^+$. The closure of the pure extension ordering follows from Lemma \ref{lemma: Iteration pure closure}.
\end{proof}

So in particular, after forcing with $\dP_{\beta}$, the tail forcing adds no new subsets of $\kappa_{\beta}$. This implies that we can define the class forcing extension as the ``limit'' of the forcing extensions by $\dP_{\beta}$ for $\beta\in\On$. Abusing notation a bit, we will refer to this class forcing extension as an extension ``by $\dP$'', treating $\dP$ as the Easton support limit of the class iteration $(\dP_{\alpha},\dot{\dQ}_{\alpha})_{\alpha\in\On}$.

\begin{mylem}\label{lemma: singular cardinals}
    After forcing with $\dP$, the class of cardinals consists of $\aleph_0$ as well as $\kappa_{\alpha}$, $\kappa_{\alpha}^+$ and $\kappa_{\alpha}^{++}$ for each $\alpha\in\On$.
\end{mylem}

\begin{proof}
    It is clear that all $V$-cardinals which are not as in the statement of the lemma are collapsed. We show that all the claimed cardinals are preserved. To this end, let $\alpha\in\On$. Since $\dot{\dP}_{\alpha,\beta}$ is forced to be a Prikry-type forcing with a ${<}\,\kappa_{\alpha}^{++}$-closed pure extension ordering and pure $\kappa_{\alpha}^{++}$-decidability, it suffices to show that $\dP_{\alpha}$ preserves $$\text{$\kappa_{\alpha}$, $\kappa_{\alpha}^+$ and $\kappa_{\alpha}^{++}$.}$$

    First assume that $\alpha$ is a successor. It follows that $\dP_{\alpha-1}$ has size ${<}\,\kappa_{\alpha}$ and thus preserves $\kappa_{\alpha}$. Let $G_{\alpha-1}$ be $\dP_{\alpha-1}$-generic. Then $\dot{\dQ}_{\alpha-1}^{G_{\alpha-1}}$ is $\kappa_{\alpha}$-cc (by standard facts for the Levy collapse or by Lemma \ref{lemma: Properties of A} for $\dA$) and thus preserves $\kappa_{\alpha}$, $\kappa_{\alpha}^+$ and $\kappa_{\alpha}^{++}$. Since $\dP_{\alpha-1}*\dot{\dQ}_{\alpha-1}=\dP_{\alpha}$, we are done.

    Now assume $\alpha$ is a limit. By the previous paragraph it follows that a cofinal set of cardinals below $\kappa_{\alpha}$ is preserved by $\dP$ and thus $\kappa_{\alpha}$ is preserved as well. Now assume in the first case that $\kappa_{\alpha}$ is singular. It follows that if $\kappa_{\alpha}^+$ is collapsed by $\dP_{\alpha}$, that poset forces $|\kappa_{\alpha}^+|=\kappa_{\alpha}$ and thus $\cf(\kappa_{\alpha}^+)<\kappa_{\alpha}$. However, this situation is not possible, since for every $\beta<\alpha$, $\dP_{\alpha}\cong\dP_{\beta}*\dot{\dP}_{\beta,\alpha}$ where $|\dP_{\beta}|<\kappa_{\alpha}$ (and thus cannot collapse $\kappa_{\alpha}^+$) and $\dot{\dP}_{\beta,\alpha}$ is a Prikry-type forcing with almost pure $\kappa_{\alpha}^+$-decidability and a ${<}\,\kappa_{\beta}^+$-closed pure extension odering (thus it cannot add cofinal subsets of $\kappa_{\alpha}^+$ of ordertype $\leq\kappa_{\beta}$). For $\kappa_{\alpha}^{++}$, simply note that because of the $\GCH$, $|\dP_{\alpha}|\leq\kappa_{\alpha}^+$.
    
    Lastly, assume in the second case that $\kappa_{\alpha}$ is regular. Then $\kappa_{\alpha}$ is in fact inaccessible and $|\dP_{\beta}|<\kappa$ for every $\beta<\kappa$. In particular, $\dP_{\alpha}$ is the direct limit of ${<}\,\kappa_{\alpha}$-sized forcing notions and thus $\kappa_{\alpha}$-cc. It follows that $\kappa_{\alpha}$, $\kappa_{\alpha}^+$ and $\kappa_{\alpha}^{++}$ are preserved.
\end{proof}
Now we turn toward showing that, in $V[\dP]$, $\AP_{\delta}$ fails for every singular cardinal $\delta$. We intend to use Theorem \ref{Theorem: Indestructibility of neg AP}. To apply this theorem, we need to find a poset which projects onto the direct extension ordering on a suitable tail forcing and forces $\SLIP$.

\begin{mydef}
    Let $\beta<\alpha$ be ordinals. Let $G_{\beta}$ be $\dP_{\beta}$-generic and work in $V[G_{\beta}]$. We define
    $$\textstyle
    \dT_{\beta,\alpha}:=\prod_{\gamma\in[\beta,\alpha)}\dT((\dot{\dP}_{\beta,\gamma}^{G_{\beta}},\dot{\leq}_{\beta,\gamma}^{G_{\beta}}),(\dot{\dQ}_{\gamma}^{G_{\beta}},\dot{\leq}_{\gamma,0}^{G_{\beta}}))$$
    where we regard $\dot{\dQ}_{\gamma}$ as a $\dP_{\beta}*\dot{\dP}_{\beta,\gamma}$-name.
\end{mydef}

By Lemma \ref{lemma: Projection Easton support Magidor} and standard arguments, one sees the following:

\begin{mylem}\label{lemma: Properties of T}
    Let $\beta<\alpha$ be ordinals. Let $G_{\beta}$ be $\dP_{\beta}$-generic and work in $V[G_{\beta}]$. The poset $\dT_{\beta,\alpha}$ is ${<}\,\kappa_{\beta}^{++}$-directed closed. Moreover, there is a projection from $\dT_{\beta,\alpha}$ onto $(\dot{\dP}_{\beta,\alpha}^{G_{\beta}},\dot{\leq}_{(\beta,\alpha),0}^{G_{\beta}})$. \qed
\end{mylem}

We show that we can actually apply Theorem \ref{Theorem: Indestructibility of neg AP} to $\dT_{\beta,\alpha}$:

\begin{mylem}\label{lemma: T forces SLIP}
    Let $\beta<\alpha$ be ordinals such that $\alpha$ is a limit. Let $\langle \delta_i\mid i<\gamma\rangle$ for some limit ordinal $\gamma$ enumerate the set $\{\kappa_j\;|\;j\in[\beta,\alpha)\text{ is odd}\}$. Then, inside $V[G_\beta]$, $\dT_{\beta,\alpha}$ forces $\SLIP(\kappa_{\beta}^{++},\langle \delta_i\mid i<\gamma\rangle)$.
\end{mylem}

\begin{proof}
    We intend to use Lemma \ref{lemma: Product forces LIP}. For each $i<\gamma$, let $\delta_i=\kappa_{j_i}$ for some odd $j_i\in(\beta,\alpha)$. If $j_i=\beta+1$ or $j_i-1$ is a limit,
    let
    $$\dP_i:=\dT((\dP_{\beta,j_i-1},\leq_{\beta,j_i-1}),(\dot{\dQ}_{j_i-1}^{G_{\beta}},\dot{\leq}_{j_i-1}^{G_{\beta}}))$$
    if $j_i-1$ is a successor and $j_i\neq\beta+1$, let
    $$\dP_i:=\dT((\dP_{\beta,j_i-1},\leq_{\beta,j_i-1}),(\dot{\dQ}_{j_i-1}^{G_{\beta}},\dot{\leq}_{j_i-1,0}^{G_{\beta}}))\times\dT((\dP_{\beta,j_i},\leq_{\beta,j_i}),(\dot{\dQ}_{j_i}^{G_{\beta}},\dot{\leq}_{j_i,0}^{G_{\beta}}))$$
    
    It follows that $\dT_{\beta,\alpha}=\prod_{i<\gamma}\dP_i$. For the application of Lemma \ref{lemma: Product forces LIP}, observe the following: If $j_i-1$ is a limit, then $(2^{\sup_{j<i}\delta_j})^+=\kappa_{j_i-1}^{++}$, since $\sup_{j<i}\delta_j=\kappa_{j_i-1}$ and the $\GCH$ holds. If $j_i-1$ is a successor, then $\sup_{j<i}\delta_j=\kappa_{j_i-2}^{++}=\delta_{i-1}^{++}$.
    
    \setcounter{myclaim}{0}

    \begin{myclaim}
        For each $i<\gamma$, $\dP_i$ is an $\LIP((2^{\sup_{j<i}\delta_j})^+,\delta_i)$-forcing (where the empty supremum is regarded as $\kappa_{\beta}$).
    \end{myclaim}

    \begin{proof}
        Let $i<\gamma$. 

        \smallskip

        \textbf{Case $j_i=\beta+1$:} Then, $\dP_i$ is equivalent to the Levy-collapse $$\Coll(\kappa_{j_i-1}^{++},<\kappa_{j_i})$$ which certainly has the claimed properties.

        \smallskip

        \textbf{Case $j_i-1$ is a limit:} Then $\dP_i=\dT((\dP_{\beta,j_i-1},\leq_{\beta,j_i-1}),(\dot{\dQ}_{j_i-1}^{G_{\beta}},\dot{\leq}_{j_i-1}^{G_{\beta}}))$. Note that $|\dP_{j_i-1}|\leq\kappa_{j_i-1}^+$ (because of the $\GCH$) and $\dot{\dQ}_{j_i-1}^{G_{\beta}}$ is a $\dP_{j_i-1}$-name for the Levy collapse $\dot{\Coll}(\kappa_{j_i-1}^{++},{<}\kappa_{j_i-1})$ and this is forced to be a $\LIP(\kappa_{j_i-1}^{++},{<}\kappa_{j_i-1})$-forcing. Thus, by Lemma \ref{lemma: Termspace of LIP is LIP}, $$\dT((\dP_{\beta,j_i-1},\leq_{\beta,j_i-1}),(\dot{\dQ}_{j_i-1}^{G_{\beta}},\dot{\leq}_{j_i-1}^{G_{\beta}}))$$ is a $\LIP(\kappa_{j_i-1}^{++},{<}\kappa_{j_i-1})$-forcing.

        \smallskip

       \textbf{Case $j_i-1$ is a successor ordinal above $\beta$.}
       Then $$\dP_i=\dT((\dP_{\beta,j_i-2},\leq_{\beta,j_i-2}),(\dot{\dQ}_{j_i-2}^{G_{\beta}},\dot{\leq}_{j_i-2,0}^{G_{\beta}}))\times\dT((\dP_{\beta,j_i},\leq_{\beta,j_i}),(\dot{\dQ}_{j_i-1}^{G_{\beta}},\dot{\leq}_{j_i-1,0}^{G_{\beta}})).$$ We have already shown that the second factor is a $\LIP(\kappa_{j_i-1}^{++},{<}\kappa_{j_i-1})$-forcing. 
       
       Regarding the first factor of the product; namely,  $$\dT((\dP_{\beta,j_i-2},\leq_{\beta,j_i-2}),(\dot{\dQ}_{j_i-2}^{G_{\beta}},\dot{\leq}_{j_i-2,0}^{G_{\beta}})):$$ 
       First, $\dot{\dQ}_{j_i-2}^{G_{\beta}}$ is a $\dP_{j_i-2}$-name for $\dot{\dA}(\kappa_{j_i-2}^{++},{<}\kappa_{j_i-1})$. Thus the first factor has size ${<}\,\kappa_{j_i}$ and is ${<}\,\kappa_{j_i-2}^{++}$-directed closed. This together easily implies that $\dP_i$ is a $\LIP(\kappa_{j_i-2}^{++},\kappa_{j_i})=\LIP((2^{\sup_{j<i}\delta_j})^+,\delta_i)$-forcing.
    \end{proof}

    In summary, by Lemma \ref{lemma: Product forces LIP}, $\dT_{\beta,\alpha}$ forces $\SLIP(\kappa_{\beta}^{++},\langle \delta_i\mid i<\gamma\rangle)$.
\end{proof}

So now we can finally show our main theorem:

\begin{proof}[Proof of Theorem \ref{theorem: AP Fails at every singular}]
    It follows from our analysis of the cardinal structure that the class of singular cardinals in $V[\dP]$ is given by the class of $\kappa_{\alpha}$ where $\alpha$ is a limit ordinal and $\kappa_{\alpha}>\alpha$. So let $\alpha$ be with that property. Fix in $V$ a normal subadditive coloring $d\colon[\kappa_{\alpha}^+]^2\to\cf(\kappa_{\alpha})$. We want to show that in $V[\dP]$, there are stationarily many points in $\kappa_{\alpha}^+$ which are not $d$-approachable (by Fact \ref{fact: eisworth handbook} this implies $\kappa_{\alpha}^+\notin I[\kappa_{\alpha}^+]$). To this end, it suffices to show that this is the case in $V[\dP_{\alpha}]$, because for any $\gamma>\alpha$, $\dot{\dP}_{\alpha,\gamma}$ is forced not to add any new subsets of $\kappa_{\alpha}^+$ by Lemma \ref{lemma: Tail properties}.    So let $\beta<\alpha$ be an even successor ordinal such that $\alpha<\kappa_{\beta-1}$. Hence, $$\dP_{\beta}\cong\dP_{\beta-1}*\dot{\dA}(\kappa_{\beta-1}^{++},\kappa_{\beta}).$$  Let $G_{\beta}$ be $\dP_{\beta}$-generic and $\dot{S}$ be a $\mathbb{P}_{\beta,\alpha}^{G_\beta}$-name for the points of $V[G_\beta]$-cofinality 
    $\kappa_{\beta-1}^{++}$ that are not $d$-approachable in $V[G_\beta][\mathbb{P}_{\beta,\alpha}^{G_\beta}]$. Then, as in Theorem~\ref{theorem: almost theorem A}, we claim:

    \begin{myclaim}
      Working in  $V[G_\beta]$, $\mathbb{P}_{\beta,\alpha}^{G_\beta}$ forces $``\dot{S}$ is stationary".
    \end{myclaim}
    \begin{proof}[Proof of clain]
      We perform a density argument. Let $p_0\in \mathbb{P}_{\beta,\alpha}^{G_\beta}$ be arbitrary.

    Let $\gamma$ be a limit ordinal and let $\langle \delta_i\mid {i<\gamma}\rangle$ enumerate  $\{\kappa_j\;|\;j\in[\beta,\alpha)\text{ is odd}\}$.\footnote{In particular, $\kappa_\alpha=\sup_{i<\gamma}\delta_i$} Let $\dot{C}$ be a $\dot{\dP}_{\beta,\alpha}^{G_{\beta}}$-name for a club subset of $\kappa_{\alpha}^+$ which, without loss of generality, is forced to be such by $p_0$.  By Lemma~\ref{lemma: Projection Easton support Magidor}, there is a projection  $$\pi\colon\dT_{\beta,\alpha}\to(\dot{\dP}_{\beta,\alpha}^{G_{\beta}},\dot{\leq}_{(\beta,\alpha),0}^{G_{\beta}}).$$
    Let $t_0\in \mathbb{T}_{\beta,\alpha}$ such that $\pi(t_0)\leq p_0$. (This exists in that $\pi$ is a projection.)

    \smallskip

    Now we apply Theorem \ref{Theorem: Indestructibility of neg AP} inside $V[G_\beta]$ to find (for $\Theta$ large enough regular) $M\prec H(\Theta)$ with $|M|<\kappa_{\beta}$ such that the following hold:
    \begin{enumerate}
        \item $M$ contains every relevant parameter (in particular, $\pi, \dot{C}$, $\dot{\dP}_{\beta,\alpha}^{G_{\beta}}$, $t_0$ and $\dT_{\beta,\alpha}$)
        \item $\cf(\sup(M\cap\kappa_{\alpha}^+))=\kappa_{\beta-1}^{++}$ and $\sup(M\cap\kappa_{\alpha}^+)$ is not $d$-approachable.
        \item \textbf{(Master condition)} There is a condition $t\in\dT_{\beta,\alpha}$, $t\leq t_0$, such that  whenever $D\in M$ is open dense in $\dT_{\beta,\alpha}$, there is $q\in D\cap M$ with $t\leq q$.
    \end{enumerate}

    Note that we can indeed invoke Theorem~\ref{Theorem: Indestructibility of neg AP}  because by Lemma \ref{lemma: Properties of T} and Lemma \ref{lemma: T forces SLIP}, $\dT_{\beta,\alpha}$ is ${<}\,\kappa_{\beta}^{++}$-directed closed and forces $\SLIP(\kappa_{\beta}^{++},\langle \delta_i\mid {i<\gamma}\rangle)$.

    \smallskip

    Let $p:=\pi(t)$. Note that $p\leq p_0$ because $\pi$ is order-preserving. We claim that $p$ forces (in  $(\dot{\dP}_{\beta,\alpha}^{G_{\beta}},\dot{\leq}_{\beta,\alpha}^{G_{\beta}})$) that $\sup(M\cap\kappa_{\alpha}^+)$ is in $\dot{C}$.  We will show this by showing that $p$ forces $\dot{C}\cap\sup(M\cap\kappa_{\alpha}^+)$ to be unbounded in $\sup(M\cap\kappa_{\alpha}^+)$. To this end, let $\zeta<\sup(M\cap\kappa_{\alpha}^+)$. We can assume $\zeta\in M\cap\kappa_{\alpha}^+$. Given any $q\in\dot{\dP}_{\beta,\alpha}^{G_{\beta}}$, $q$ forces that there exists some element of $\dot{C}$ above $\zeta$. By the almost pure $\kappa_{\alpha}^+$-decidability of $\dot{\dP}_{\beta,\alpha}^{G_{\beta}}$ (see Lemma \ref{lemma: Tail properties}), there exists $r\dot{\leq}_{(\beta,\alpha),0}^{G_{\beta}}q$ and $\eta<\kappa_{\alpha}^+$ such that $r$ forces that there exists some element of $\dot{C}$ in the interval $(\zeta,\eta)$. Ergo, the set $D$ of all $r\in\dot{\dP}_{\beta,\alpha}^{G_{\beta}}$ such that for some $\eta_r$, $r$ forces $\dot{C}\cap(\zeta,\eta_r)$ to be nonempty, is open dense in $(\dot{\dP}_{\beta,\alpha}^{G_{\beta}},\dot{\leq}_{(\beta,\alpha),0}^{G_{\beta}})$. Since $\pi$ is a projection, $\pi^{-1}[D]$ is open dense in $\dT_{\beta,\alpha}$. Ergo, by Clause~(3) above (i.e., $t$ is a master condition) there exists $r\in\pi^{-1}[D]\cap M$ with $t\leq r$ and thus $\pi(r)\in D\cap M$ with $\pi(t)\leq\pi(r)$. The corresponding $\eta_{\pi(r)}$ is in $M$ as well by elementarity. Ergo $\pi(t)$ forces that there is some element of $\dot{C}$ above $\zeta$ and below $\sup(M\cap\kappa_{\alpha}^+)$. Since $\zeta$ was arbitrary, $\pi(t)$ forces that $\dot{C}$ is unbounded in $\sup(M\cap\kappa_{\alpha}^+)$. In addition,  the closure of $\dot{\dP}_{\beta,\alpha}^{G_{\beta}}$ in the pure extension ordering implies that $\sup(M\cap\kappa_{\alpha}^+)$ is  not $d$-approachable after forcing with that poset. Ergo, $p$ forces $``\sup(M\cap \kappa_\alpha^+)$ is not $d$-approachable". Thus, $p$ forces $``\dot{S}\cap \dot{C}\neq \emptyset"$, as needed.
    
    \end{proof}
 We are done with the proof of the theorem.  
\end{proof}

\begin{myrem}
    It follows easily from our proof that, after forcing with $\dP$, for every singular cardinal $\kappa_{\gamma}$ and every even $\beta<\gamma$ with $\alpha<\kappa_{\beta-1}$, the set $\kappa_{\gamma}^+\cap\cof(\kappa_{\beta-1}^{++})$ is not a member of $I[\kappa_{\gamma}^+]$. In particular, there are stationarily many non-approachable points of arbitrarily large cofinality. This disposes with the moreover assertion in Theorem~\ref{ThmD}.
\end{myrem}

\section{Discussion and Open Questions}

There is an interesting contrast between the results of this paper and \cite{JakobTotalFailure}. Namely, the poset which we devised in this paper to collapse the successor of a singular cardinal without making it approachable has much stronger requirements than the corresponding one from \cite{JakobTotalFailure}. Due to this, we had to separate our collapses and thus were not able to obtain the failure of the approachability property for all possible cofinalities below the singular cardinal, simultaneously. Thus we ask:

\begin{myque}\label{Question 1}
    Suppose that $\lambda$ is a singular cardinal of uncountable cofinality. Is there a model of $\ZFC$ plus ``For every regular $\mu\in(\cf(\lambda),\lambda)$ there is a stationary subset of $\lambda^+\cap \cof(\mu)$ which is not in $I[\lambda^+]$"?
\end{myque}

Another relevant question regards the failure of the SCH. Answering a question from \cite{BenNeriaLambieHansonUngerDiagSCRadin}, in \cite{GitikAP} Gitik produces a model of $\ZFC$ where both $\AP$ and $\SCH$ fail on a club class.  A natural question is if this configuration can be combined with Theorem~\ref{ThmD}. More precisely we ask:
\begin{myque}
    Assuming a proper class of supercompact cardinals, is it consistent with $\mathrm{ZFC}$ that the $\mathrm{SCH}$ fails on a club class and $\AP$ fails at every singular cardinal?
\end{myque}
Combining Theorem~\ref{ThmD} with a \textbf{global} failure of the SCH is a major open problem for it would yield a negative solution to the following long-standing question:

\begin{myque}[Woodin, 80's]\label{que: Woodins}
  Let $\kappa$ be a singular strong limit cardinal. Does the approachability property $\AP_\kappa$ imply the $\SCH$ at $\kappa$?
\end{myque}

\smallskip

By a result of Jensen in the model of Theorem~\ref{ThmD} there are no special $\kappa^+$-Aronsjzan trees for no singular cardinal $\kappa$. Thus, it is natural to ask:
\begin{myque}
    Assuming appropriate large cardinals, is it possible to have a model of $\mathrm{ZFC}$ where the tree property holds at the successor of every singular cardinal? 
\end{myque}

\section*{Acknowledgments}
The second author acknowledges support from the Center of Mathematical Sciences and Applications and the Department of Mathematics at Harvard University.

\bibliographystyle{alpha} 
\bibliography{bibliography}

\end{document}